\documentclass[10pt, leqno, a4paper]{amsart}

\usepackage{amsmath, amsthm}
\usepackage{mathrsfs}    
\usepackage{amssymb}

\usepackage{graphicx,color}

\usepackage[colorlinks]{hyperref}
\hypersetup{linkcolor=blue,citecolor=blue,filecolor=black,urlcolor=blue}

\usepackage{verbatim} 

\usepackage[showonlyrefs]{mathtools}
\mathtoolsset{showonlyrefs=true}

\usepackage[sort,nocompress]{cite}

\newtheorem{theorem}{Theorem}[section]
\newtheorem{proposition}[theorem]{Proposition}
\newtheorem{corollary}[theorem]{Corollary}
\newtheorem{lemma}[theorem]{Lemma}

\theoremstyle{definition}
\newtheorem{definition}[theorem]{Definition}
\newtheorem{remark}[theorem]{Remark}

\usepackage[
  hmarginratio={1:1},     
  vmarginratio={1:1},     
  textwidth=15cm,        
  textheight=20cm,  
  heightrounded,          
]{geometry}

\numberwithin{equation}{section}

\def\eps{\varepsilon}
\def\vphi{\vphi}

\def\Ecal{\mathcal{E}}

\def\Kcal{\mathcal{K}}

\def\Mcal{\mathcal{M}}
\def\Ncal{\mathcal{N}}

\def\Hcal{\mathcal{H}}

\def\N{\mathbb{N}}

\def\R{\mathbb{R}}

\def\S{\mathbb{S}}

\def\hbar{\bar{h}}

\def\leq{\leqslant}
\def\geq{\geqslant}

\newcommand{\pa}{\partial}

\newcommand{\rst}[1]{\ensuremath{{\mathbin |}%
\raise-.5ex\hbox{$#1$}}}
\newcommand{\mf}[1]{\mathbf{#1}}

\newcommand{\loc}{\mathrm{loc}}

\DeclareMathOperator{\rad}{rad}

\definecolor{verde}{rgb}{0,0.35,0.1} 
\definecolor{rosso}{rgb}{0.7,0,0}
\definecolor{blue}{rgb}{0,0,1}
\definecolor{viola}{rgb}{0.6,0,0.4}

\begin{document}

\title[Existence and symmetry of les for Schr\"odinger systems]{New existence and symmetry results for least energy positive solutions of Schr\"odinger systems with mixed competition and cooperation terms}

\author[N. Soave]{Nicola Soave}\thanks{}
\address{Nicola Soave \newline \indent
Mathematisches Institut, Justus-Liebig-Universit\"at Giessen, \newline \indent
Arndtstrasse 2, 35392 Giessen (Germany)}
\email{nicola.soave@gmail.com; nicola.soave@math.uni-giessen.de.}

\author[H. Tavares]{Hugo Tavares}\thanks{}
\address{Hugo Tavares \newline \indent  Center for Mathematical Analysis, Geometry and Dynamical Systems \newline \indent Mathematics Department \newline \indent Instituto Superior T\'ecnico, Universidade de Lisboa \newline \indent Av. Rovisco Pais, 1049-001 Lisboa, Portugal}
\email{htavares@math.ist.utl.pt}


\thanks{{Acknowlegments.}  H. Tavares is supported by Funda\c c\~ao para a Ci\^encia e Tecnologia through the program \textit{Investigador FCT} and the project PEst-OE/EEI/LA0009/2013. N. Soave and H. Tavares are partially supported through the project ERC Advanced Grant  2013 n. 339958 ``Complex Patterns for Strongly Interacting Dynamical Systems - COMPAT''}

\date{\today}

\begin{abstract}
In this paper we focus on existence and symmetry properties of solutions to the cubic Schr\"odinger system
\[
-\Delta u_i +\lambda_i u_i = \sum_{j=1}^d \beta_{ij} u_j^2 u_i \quad   \text{in $\Omega \subset \R^N$},\qquad i=1,\dots d
\]
where $d\geq 2$, $\lambda_i,\beta_{ii}>0$, $\beta_{ij}=\beta_{ji}\in \R$ for $j\neq i$, $N=2,3$. The underlying domain $\Omega$ is either bounded or the whole space, and $u_i\in H^1_0(\Omega)$ or $u_i\in H^1_{\rad}(\R^N)$ respectively. We establish new existence and symmetry results for least energy positive solutions in the case of mixed cooperation and competition coefficients, as well as in the purely cooperative case.
\end{abstract}

\maketitle

\section{Introduction}

The existence and the qualitative description of \emph{least energy solutions} to the nonlinear elliptic system
\begin{equation}\label{system 2 comp}
\begin{cases}
-\Delta u + \lambda_1 u = \mu_1 u^3 + \beta u v^2 \\
-\Delta v + \lambda_2 v = \mu_2 v^3 + \beta u^2 v \\
u, v \in H_0^1(\Omega),
\end{cases}  \quad \text{with $\Omega \subset \R^N$ or $\Omega=\R^N$, and $N=2,3$},
\end{equation}
have attracted considerable attention in the last ten years, starting from the seminal paper \cite{LinWei} by T.-C. Lin and J. Wei. Collecting all the results contained in several contributions, it is possible to obtain an exhaustive picture of the problem, see the forthcoming Subsection \ref{sub:review}. In striking contrast, a complete understanding in the case of an arbitrary $d \ge 3$ components system
\begin{equation}\label{system}
\begin{cases}
-\Delta u_i +\lambda_i u_i = \sum_{j=1}^d \beta_{ij} u_j^2 u_i & \text{in $\Omega$}\\
u_i \not \equiv 0 \\
u_i \in H_0^1(\Omega),
\end{cases} \qquad i=1,\dots,d, \quad \beta_{ij}= \beta_{ji}
\end{equation}
is not available, mainly due to the possible coexistence of \emph{cooperation} and \emph{competition}, that is, the existence of two pairs $(i_1,j_1)$ and $(i_2,j_2)$ such that $\beta_{i_1 j_1}>0$ and $\beta_{i_2 j_2} < 0$. We recall that the sign of the coupling parameter $\beta_{ij}$ determines the nature of the interaction between the components $u_i$ and $u_j$: if $\beta_{ij}>0$, then they cooperate, while if $\beta_{ij}<0$, then they compete. Very recently, the systematic study of existence of least energy solutions in problems with simultaneous cooperation and competition has been started by the first author in \cite{Soave} and by Y. Sato and Z.-Q. Wang in \cite{SatoWang}. Nevertheless, there are still some gaps to fill towards a complete understatement of the problem. In the present paper we give a contribution to fill some of these gaps, and, in the mean time, we analyse the symmetry properties of least energy solutions to \eqref{system}, proving results which are new also in a purely cooperative context ($\beta_{ij} > 0$ for every $i \neq j$), and recover what is known in the purely competitive one ($d=2$ and $\beta_{12} =\beta< 0$). 

In order to motivate our research, in the following we review the results already available in the literature, but before it is worth to observe that thanks to the assumption $\beta_{ij}=\beta_{ji}$, system \eqref{system} has variational structure, as its solutions are critical points of the functional $J:H_0^1(\Omega,\R^d) \to \R$ defined by 
\begin{equation}\label{def: functional}
J(\mf{u}):= \int_{\Omega} \frac{1}{2} \sum_{i=1}^d \left( |\nabla u_i|^2 + \lambda_i u_i^2 \right) - \int_\Omega\frac{1}{4} \sum_{i,j=1}^d \beta_{ij} u_i^2 u_j^2,
\end{equation}
where we used the vector notation $\mf{u}=(u_1,\dots,u_d)$. Observe that \eqref{system} admits \emph{semi-trivial solutions}, i.e., solutions $\mf{u}\not\equiv 0$ with some zero components. However, we will be only interested in the existence of	\emph{positive solutions}: $\mf{u}$ solving \eqref{system} such that $u_i>0$ for every $i$. In particular, we will be interested in the existence of \emph{least energy positive solutions}, that is solutions achieving the \emph{least energy positive level};
\[
\inf \left\{J(\mf{u}):\  \mf{u} \text{ is a solution of \eqref{system} such that } u_i>0 \text{ for all }i \right\}.
\]
Observe that, due to the shape of the functional and to the strong maximum principle, in the definition of the previous energy level, one can replace $u_i>0$ by $u_i\not \equiv 0$. Observe moreover that, depending on the ranges of $\beta_{ij}$, this might not coincide with the least energy level (ground state):
\[
\inf\{J(\mf{u}):\ \mf{u} \neq 0 \text{ is a solution of }  \eqref{system}  \},
\]
(see for instance \cite{AmbrosettiColorado} or \cite{Sirakov}), which is an additional difficulty when one looks for least energy positive solutions.

\subsection{Known results}\label{sub:review}
Let us first describe the existing results in either the \emph{purely cooperative case} $\beta_{ij}>0$ for every $i\neq j$, or in the \emph{purely competitive case} $\beta_{ij}<0$ for every $i\neq j$. Some results deal with $\Omega$ bounded, while others with the case $\Omega=\R^N$. An important observation is that, in all the cited contributions dealing with the case $\R^N$, one is naturally led to work in $H^1_\textrm{rad}(\R^N)$, and with least energy positive \emph{radial} solutions, that is positive radial solutions having minimal energy among all the positive radial solutions. In fact, in the purely cooperative case, each positive solution of \eqref{system} is radially decreasing, as comes out from \cite{BuscaSirakov}; hence a least energy positive level coincides with the radial one. On the other hand, in the purely competitive case, it is proved in \cite[Theorem 1]{LinWei} that the least energy positive level is not achieved, so one is naturally led to deal with the radial one. Since working in a radial setting makes possible to take advantage of the compactness of the Sobolev embedding $H^1_{\rad}(\R^N) \hookrightarrow L^4(\R^N)$, the research of least energy positive radial solutions in $\R^N$ is substantially equivalent to the research of least energy solutions in bounded domains. For the sake of clarity, in what follows we always refers to a result as it was stated in its original contribution, but the reader has always to keep in mind that, whenever we cite a result for ``least energy radial positive solution in $\R^N$", this also yields an existence of ``least energy positive solution in a bounded domain", and vice-versa.

\medbreak

Having this in mind, we focus at first on the $2$ components system \eqref{system 2 comp} in $\R^N$ ($N=2,3$), with $\lambda_{1},\lambda_2,\mu_1,\mu_2 >0$. By collecting the main theorems in \cite{AmbrosettiColorado, LinWeiErratum,LinWei, MaiaMontefuscoPellacci, Sirakov}, one deduces that there exist $0 < \underline{\beta} \le \overline{\beta}$ (depending on $\lambda_i$ and $\mu_i$) such that if either $-\infty<\beta < \underline{\beta}$, or $\beta > \overline{\beta}$, then \eqref{system 2 comp} has a least energy radial positive solution (whose level coincides with the least energy positive level for $\beta>0$). For the expression of the optimal values for $\underline{\beta}$ and $\overline{\beta}$, we refer to \cite{ChenZou, Mandel, Sirakov}.

The results for the $2$ components system have been partially extended for systems with an arbitrary number of components. Sufficient conditions for the existence of a least energy radial positive solutions of \eqref{system} in $\R^N$ are the following.
\begin{itemize}
\item \emph{Strong cooperation}: $\lambda_1=\dots=\lambda_N=\lambda>0$, $\beta_{ii}>0$, and $\beta_{ij}=\beta$ for every $i \neq j$ is larger than a positive constant depending on $\beta_{ii}$ and $\lambda$ (see Corollary 2.3 and Theorem 2.1 in \cite{LiuWang}; see also Theorem 1.6 and Remark 3 in \cite{Soave}). 
\item \emph{Weak cooperation}: $\lambda_i >0$, $\beta_{ii}>0$, $0<\beta_{ij}$ smaller than a positive constant depending on $\lambda_i$ and $\beta_{ii}$, and the matrix $(\beta_{ij})$ is positive definite (see Theorem 2 in \cite{LinWeiErratum});
\item \emph{Competition}: if $\lambda_i >0$, $\beta_{ii}>0$, and $\beta_{ij} \le 0$ for every $i \neq j$, then there exists a least energy radial positive solution (see Theorem 1.1. plus Remark 1.5 in \cite{LinWei2} for the case $\Omega$ bounded; we refer also to Theorem 3.1 in \cite{LiuWang}, and to Corollary 1.4 plus Proposition 1.5 in \cite{Soave}). 
\end{itemize}
Other sufficient conditions in a purely cooperative setting have been given in  \cite[Section 4]{Sirakov}, \cite[Theorem 2.1]{LiuWang} and \cite{Chang}. It is natural to assume that $\beta_{ij}$ is either large, or small, with respect to $\beta_{ii}$ and $\beta_{jj}$. Indeed if for instance $\beta_{ii} \le \beta_{ij} \le \beta_{jj}$ and $\lambda_i>\lambda_j$, then a positive solution of \eqref{system} does not exists, see Theorem 1-($ii$) in \cite{Sirakov} or Theorem 0.2 in \cite{BartschWang}. 

\medbreak

As far as the possible occurrence of simultaneous cooperation and competition is concerned, in \cite[Theorem 0.1]{SatoWang} Y. Sato and Z.-Q. Wang considered a $3$ components system in a bounded domain, showing that a least energy positive solution of \eqref{system} does exist if $\beta_{13},\beta_{23} \le 0$ are fixed, and $\beta_{12} \gg 1$ is very large (depending on $\beta_{13}$ and $\beta_{23}$); we refer to this cases as \emph{competition vs. arbitrarily large cooperation}. In \cite[Theorems 1.6, 1.7, 1.9]{Soave} the author considered an arbitrary $d$ components system, proving the existence of least energy positive solutions whenever the $d$ components are divided into $m$ groups, with $m \le d$, and 
\begin{itemize}
\item the relation between components of the same group is purely cooperative, with coupling parameters greater than an explicit positive constant,
\item the relation between components of different groups is competitive, and the competition is very strong.
\end{itemize}
When restricted to a $3$ components system, this leads for instance to existence of a least energy solution if $\beta_{12}> \overline{\beta}>0$, and $\beta_{13},\beta_{23} \ll -1$ (depending on $\beta_{12})$. We refer to this cases as to \emph{strong cooperation vs. arbitrarily large competition}. In the previous two results we would like to stress that the ``large'' parameters depend on \emph{all} the other interaction terms. Ahead we will give a result which allows to fix \emph{a priori} all the ranges for the parameters, which consists of a novelty when dealing with mixed cooperative and competitive interaction.

\begin{remark}
It is worth to point out the difference between \emph{strong cooperation} and \emph{arbitrarily large cooperation}: in the former case, we mean that some $\beta_{ij}$'s are greater than a positive constant which can be large but is fixed and determined as function of $\beta_{ii}$ and $\lambda_i$, more or less explicitly; in the latter one we mean that some $\beta_{ij}$ have to be thought as very large parameters which are tending to $+\infty$, depending in a non explicit way on the other parameters. The same discussion holds for the distinction between \emph{competition} and \emph{arbitrarily large competition}.
\end{remark}

For further existence results for system \eqref{system} with mixed cooperative and competitive couplings, which regard solutions not necessarily of least energy, we refer the reader to \cite[Theorem 4]{LinWei}, \cite[Theorem 2.1]{LiuWang2008}, \cite[Corollary 1.4]{Soave}, \cite{SatoWang2} and \cite{Colorado}.

Concerning the symmetry properties of least energy positive solutions in bounded domains, the main results are contained in \cite{TavaresWeth,WangWillem}. We postpone a precise description of them after having introduced some notation. In the next subsections, we describe the main results of this paper.

\subsection{Main results: existence}\label{sub: existence}
We are concerned with the existence of least energy solutions of system \eqref{system}:
\[
\begin{cases}
-\Delta u_i +\lambda_i u_i= \sum_{i=1}^d \beta_{ij} u_i u_j^2 & \text{in $\Omega$} \\
u_i=0 & \text{on $\partial \Omega$},
\end{cases} \qquad i=1,\dots,d
\]
where either
\begin{equation}\label{basic assumption}
\begin{split}
\text{$\Omega$ is a bounded domain of $\R^N$ with $N=2,3$}, \\ \text{$\lambda_i>-\mu_1(\Omega)$ and $\beta_{ii}>0$ for every $i=1,\dots,d$}, \\
\text{$\beta_{ij}=\beta_{ji} \in \R$ for every $i \neq j$},
\end{split}
\end{equation} 
and $\mu_1(\Omega)$ is the first eigenvalue of the Laplace operator with homogeneous Dirichlet boundary conditions on $\Omega$, or
\begin{equation}\label{basic assumption in space}
\begin{split}
\text{$\Omega=\R^N$, with $N=2,3$}, \\ \text{$\lambda_i>0$ and $\beta_{ii}>0$ for every $i=1,\dots,d$}, \\
\text{$\beta_{ij}=\beta_{ji} \in \R$ for every $i \neq j$}.
\end{split}
\end{equation}
In this last case, the boundary condition $u_i=0$ on $\pa \Omega$ has to be replaced by $u_i\to 0$ as $|x|\to \infty$, and in the following instead of $H_0^1(\Omega)$ we have to write $H^1_{\rad}(\R^N)$, the space of $H^1$ radially symmetric functions in $\R^N$. 

In the following we recall some notations already introduced in 
\cite{Soave}.

\begin{itemize}
\item $B:= (\beta_{ij})_{i,j=1,\dots,d}$, and we refer to it as to the \emph{coupling matrix} of system \eqref{system}. 
\item We endow the Sobolev space $H_0^1(\Omega)$ - or $H^1_{\rad}(\R^N)$ - with scalar products and norms
\[
\langle u, v \rangle_i := \int_{\Omega} \left( \nabla u \cdot \nabla v + \lambda_i uv\right)  \quad \text{and} \quad \|u\|_{\lambda_i}^2= \|u\|_i^2:=\langle u, u \rangle_i,
\]
for every $i=1,\dots,d$; in light of the assumptions on $\lambda_i$, these norms are equivalent to the standard one. 
\item For an arbitrary $1 \le m \le d$, we say that a vector $\mf{a}=(a_0,\dots,a_m) \in \N^{m+1}$ is \emph{a $m$-decomposition of $d$} if
\[
0=a_0<a_1<\dots<a_{m-1}<a_m=d;
\]
given a $m$-decomposition $\mf{a}$ of $d$, we set, for $h=1,\dots,m$,
\begin{equation}\label{def Ih}
\begin{split}
& I_h:= \{i \in  \{1,\dots,d\}:  a_{h-1} < i \le a_h \}, \\
&\mathcal{K}_1:= \left\{(i,j) \in I_h^2 \text{ for some $h=1,\dots,m$, with $i \neq j$}\right\}, \\ &\mathcal{K}_2 := \left\{(i,j) \in I_h \times I_k \text{ with $h \neq k$} \right\}. 
\end{split}
\end{equation}
\end{itemize}

This way, we have partitioned the set $\{1,\ldots, d\}$ into $m$ groups $I_1,\ldots, I_m$, and have consequently splitted the components into $m$ groups: $\{u_i:\ i\in I_h\}$. 

We point out that the limit cases $m=1$ or $m=d$ are also included in our terminology. This means that we will be able to recover (and sometimes improve) the known results for the purely cooperative case (taking $m=1$) and for the purely competitive or weakly cooperative one (taking $m=d$).

Let us continue to introduce more notations.

Given a $m$-decomposition $\mf{a}$ of $d$, we introduce the Nehari-type set induced by $\mf{a}$ as
\begin{equation}\label{def N}
\Ncal := \left\{ \mf{u} \in H^1_0(\Omega;\R^d) \left|\  
\begin{array}{l}
\sum_{i \in I_h} \|u_i\|_i \neq 0 \text{ and } \sum_{i \in I_h} \partial_i J (\mf{u}) u_i=0 \\
\text{for every } h=1,\dots,m
\end{array}\right. \right\}.
\end{equation}

\medbreak

 Only to fix our minds, we assume from now on that $\Omega$ is a bounded domain, and then we suppose that \eqref{basic assumption} is in force. Unless otherwise specified, the results can be extended for systems in $\R^N$ replacing ``least energy positive solution" with ``least energy radial positive solution".

For a given $d \ge 2$, let $\mf{a}$ be a $m$ decomposition of $d$. We set
\[
c:= \inf_{\mf{u} \in \mathcal{N}} J(\mf{u}),
\]
the infimum of $J$ on the Nehari-type set $\mathcal{N}$. 

\begin{theorem}\label{thm:c_B_is_achieved}
There exists $K>0$, depending only on $\beta_{ii}, \lambda_i$ ($i=1,\ldots, d$), such that, whenever $B$ satisfies
\begin{align*}
&\beta_{ij} \geq 0 \quad \forall (i,j)\in \Kcal_1, \qquad -\infty <\beta_{ij}< K \quad \forall (i,j)\in \Kcal_2,
\end{align*}
then $c$ is achieved by a nonnegative $\mf{u}_{\min} \in \Ncal$. Furthermore, any minimizer is a nonnegative solution of \eqref{system}.
\end{theorem}
This is an improvement of Theorem 1.3 in \cite{Soave}, where $K$ depends also on $\beta_{ij}$ with $(i,j) \in \mathcal{K}_1$, and the minimization is considered in an \emph{open subset of $\mathcal{N}$} (notice that in such case a minimizer with all positive components needs not to be a least energy solution). 

Theorem \ref{thm:c_B_is_achieved} gives existence of nonnegative solutions for systems of $d$ equations where the populations are associated in groups, in such a way that inside each group there is cooperation, while between different groups we have either competition or weak cooperation.

Observe that if we can show that $\mf{u}_{\min}$ has all positive components, we can immediately conclude that it is a least energy positive solution of the system. We are able to obtain this conclusion in several situations. First, applying Theorem \ref{thm:c_B_is_achieved} in the particular case $m=d$, which leads to the decomposition $\mf{a}=(0,1,\ldots, d)$, we have existence of least energy positive solutions of \eqref{system} also in regimes of \emph{competition and/or weak cooperation}.

\begin{corollary}\label{corol: comp e weak coop}
There exists $K>0$, depending only on $\beta_{ii}, \lambda_i$ ($i=1,\ldots, d$), such that if 
\[
-\infty< \beta_{ij} < K \qquad \text{for every $i \neq j$},
\]
then $c$ is achieved by a least energy positive solution of \eqref{system}.
\end{corollary}

Recall from the previous subsection that, up to now, this result was only know in the pure competitive or in the pure cooperative cases. Focusing on this last situation, note in particular that, unlike in \cite[Theorem 2]{LinWeiErratum} (where the case $\beta_{ij}>0$ for every $i \neq j$ is considered), Corollary \ref{corol: comp e weak coop} does not require the positive definiteness of the matrix $B$. Moreover, in the particular case of \emph{pure cooperation}, we will show that
\begin{equation}\label{eq:special_shape_of_K}
K:= \frac{\min_{i=1,\ldots, d} \{S_i^2\}}{2\sum_{j=1}^m \frac{S_j^2}{\beta_{jj}}},\quad  \text{ where } S_i:=\inf_{\int_\Omega u^4=1} \|u\|_i^2,
\end{equation}
which generalizes \cite{WangWillem}, where only the case $d=2$ is considered.

If we consider a general $m$ decomposition with $m<d$, to find new least energy positive solutions we have to find conditions on the coupling parameters ensuring that the minimizer $\mf{u}_{\min}$ in Theorem \ref{thm:c_B_is_achieved} has all non-trivial components. In what follows we shall use this argument to prove new existence results with respect to those in \cite{SatoWang, Soave}. As in the quoted papers, the idea is to find conditions on the coupling parameters $\beta_{ij}$ which ensure that 
\[
\inf_{\mathcal{N}} J < \inf\left\{J(\mf{u}): \text{$\mf{u} \in \mathcal{N}$ and $u_i =0$ for some $i$} \right\}.
\]  

\begin{theorem}\label{thm: weak coop and strong coop 1}
Let $d \ge 2$, let $\mf{a}$ be a $m$-decomposition of $d$ for some $1 \le m \le d$. Let $K$ be the constant defined in Theorem \ref{thm:c_B_is_achieved}. If 
\begin{itemize}
\item[($i$)] $\beta_{ij}=\beta_h > \max\{\beta_{ii}: i \in I_h\}$ for every $(i,j) \in I_h^2$ with $i \neq j$, $h=1,\dots,m$;
\item[($ii$)] $\beta_{ij}=b < K$ for every $(i,j) \in \mathcal{K}_2$;
\item[($iii$)] $\lambda_i=\lambda_h>-\mu_1(\Omega)$ for every $i \in I_h$, $h=1,\dots,m$;
\end{itemize} 
then any minimizer of $J$ in $\mathcal{N}$ is positive, and hence system \eqref{system} has a least energy positive solution.
\end{theorem}

 This statement together with Corollary 2.3 in \cite{LiuWang} and our Corollary \ref{corol: comp e weak coop} provides the natural extension of what is known for the $2$ components system for systems with an arbitrary number of equations. To be more precise, let us give an example in the case of $3$ components system with the additional assumptions $\lambda_1=\lambda_2=\lambda$ and $\beta_{13}=\beta_{23}$. In this case, from our result and the quoted ones we deduce the existence of $0<\underline{\beta}<\overline{\beta}$ such that \eqref{system} admits a least energy positive solution when one of the following conditions is verified:
\begin{equation}\label{bar b_nderbar_b}
\begin{split}
& \beta_{12}=\beta_{13}=\beta_{23}>\overline{\beta};\quad \lambda_3=\lambda\\
& \beta_{12}>\overline{\beta}\text{ and }-\infty<\beta_{13}=\beta_{23}<\underline{\beta};\\
& -\infty<\beta_{12},\beta_{13}=\beta_{23}<\underline{\beta}.
\end{split}
\end{equation}
The downsize of the previous theorem is that the restriction $\beta_{ij}=b$ for every $(i,j) \in \mathcal{K}_2$ is quite strong. For this reason we present also an alternative result, which permits to avoid this assumption but requires that $|\beta_{ij}|$ is not too large for $(i,j) \in \mathcal{K}_2$.

\begin{theorem}\label{thm: weak coop and strong coop 2}
Let $d \ge 2$, let $\mf{a}$ be a $m$-decomposition of $d$ for some $0 \le m \le d$. Let $K$ be the constant defined in Theorem \ref{thm:c_B_is_achieved}, and fix $\alpha>1$. If
\begin{itemize}
\item[($i$)] $\beta_{ij}=\beta_h$ for every $(i,j) \in I_h^2$ with $i \neq j$, $h=1,\dots,m$, and
\[
\beta_h > \frac{\alpha}{\alpha-1}\max\{\beta_{ii}: i \in I_h\};
\]
\item[($ii$)] for every $(i,j) \in \mathcal{K}_2$ there holds 
\[
|\beta_{ij}| \le \frac{K}{\alpha d^2};
\]
\item[($iii$)] $\lambda_i=\lambda_h>-\mu_1(\Omega)$ for every $i \in I_h$, $h=1,\dots,m$;
\end{itemize} 
then any minimizer of $J$ in $\mathcal{N}$ is positive, and hence system \eqref{system} has a least energy positive solution.
\end{theorem}
We observe that the previous two results seem to be the first dealing simultaneously with strong and weak cooperation.

With these results in hands, together with those in \cite{LiuWang,SatoWang,Soave}, we can give quite a complete picture for the problem of the existence of a least energy solution for system \eqref{system} when $d \ge 1$. We have already recalled in Subsection \ref{sub:review} that the existence of a least energy positive solution has been proved in regimes of \emph{strong cooperation}, \emph{competition}, \emph{weak cooperation}, \emph{arbitrarily large cooperation vs. competition}, \emph{strong cooperation vs. arbitrarily large competition}. Thanks to Corollary \ref{corol: comp e weak coop}, Theorems \ref{thm: weak coop and strong coop 1} and \ref{thm: weak coop and strong coop 2}, we have existence results in regimes of \emph{competition or weak cooperation} and \emph{weak cooperation or competition vs. strong competition}. Recalling the non-existence of positive solution when $\beta_{ii} \le \beta_{ij} \le \beta_{jj}$ proved in \cite{BartschWang, Sirakov}, at least from a qualitative point of view the existence of a least energy positive solutions is proved in all the admissible cases.  

We would like to stress that in all the previous theorems, unlike in \cite{SatoWang,Soave}, the bounds only depend on $\lambda_i$ and $\beta_{ii}$, for $i=1,\ldots,d$.

\medskip

\subsection{Main results: non existence of least energy positive solution in $\R^N$}\label{subsect:nonexistence_inR^N}

As it was observed in Subsection \ref{sub:review}, in the purely competitive case ($\beta_{ij}<0$ for every $i\neq j$), a least energy positive solution in $\R^N$ (note that here the adjective ``radial" is omitted) is not achieved, cf. Theorem 1 in \cite{LinWei}. We have already recalled that in regimes of pure and strong cooperation, least energy solutions in $\R^N$ does exists, and are naturally radially symmetric and radially decreasing with respect to some point. A case of mixed cooperation and competition is treated in \cite[Theorem 3]{LinWeiErratum}, but only when exactly one  state repels all the other, and the remaining ones have small attractive coefficients. The general case is left open in \cite{LinWei,LinWeiErratum}. 

As we pointed out, our Theorem \ref{thm:c_B_is_achieved} in case $\Omega=\R^N$ works only in a radial setting, and allows much more combinations of cooperation with competition coefficients. It seems natural to ask whether this restriction, in general, is completely justified or not. We can prove that it is in several situations.

\begin{theorem}\label{thm: les in space comp}
Under \eqref{basic assumption in space}, let $d\geq 2$, and let $\mf{a}$ be a $m$-decomposition of $d$, with $1 \le m \le d$. If 
\begin{itemize}
\item $\beta_{ij} \ge  0 \quad \forall (i,j) \in \mathcal{K}_1$,
\item $\beta_{ij}  \le 0 \quad \forall (i,j) \in \mathcal{K}_2$, and there exist $h_1\neq h_2$ such that $\beta_{ij}<0$ for every $(i,j)\in I_{h_1}\times I_{h_2}$;
\end{itemize}
then
\begin{equation}\label{def di l}
l:= \inf_{\mathcal{M}} J
\end{equation}
is not achieved, where
\begin{equation}\label{def: Nehari non radiale}
\mathcal{M}:= \left\{ \mf{u} \in H^1(\R^N,\R^d) \left|\  
\begin{array}{l}
\sum_{i\in I_h} \|u_i\|_i \neq 0 \text{ and } \sum_{i \in I_h} \partial_i J (\mf{u}) u_i=0 \\
\text{for every } h=1,\dots,m
\end{array}\right. \right\}.
\end{equation} 
\end{theorem}

For the reader's convenience, we recall that $I_h$, $\mathcal{K}_1$ and $\mathcal{K}_2$ have been defined in \eqref{def Ih}. Notice also that $\mathcal{M}$ is the Nehari set associated to the $m$-decomposition $\mf{a}$ \emph{without the radial constraint}.

\subsection{Main results: partial symmetry}

We now pass to the statements regarding partial symmetry. First, we recall the following.

\begin{definition}
Let $\Omega \subset \R^N$ be radial with respect to $0$. A function $u: \Omega \to \R$ is called \emph{foliated symmetric Schwarz} with respect to the direction $p \in \S^{N-1}$ if $u$ depends only on $(r,\theta):= (|x|, \arccos( x\cdot p / |x|) )$, and is non-increasing in $\theta$. \\
We write that the vector valued function $(u_1,\dots,u_k)$ is foliated Schwarz symmetric with respect to $p$ if each $u_i$ is foliated Schwarz symmetric with respect to $p$.\\
We write that $(u_1,\dots,u_h)$ and $(u_{h+1},\dots,u_k)$ are foliated Schwarz symmetric with respect to antipodal directions if there exists $p \in \S^{N-1}$ such that $u_1,\dots,u_h$ are foliated Schwarz symmetric with respect to $p$, while $u_{h+1},\dots,u_k$ are foliated Schwarz symmetric with respect to $-p$.
\end{definition}

We analyse the symmetry properties of $\mf{u}_{\min}$ when $\Omega$ is a bounded radial domain (since in $\R^N$ we deal with radial solutions, an analogue statement would be trivial in that setting). The following result is stated in the greatest possible generality. 

\begin{theorem}\label{thm: partial symmetry}
Let $\Omega \subset \R^N$ be a bounded radially symmetric domain, with $N =2,3$. Let $d \ge 2$, let $\mf{a}=(a_0,\ldots, a_m)$ be a $m$-decomposition of $d$ for some $1 \le m \le d$, assume that \eqref{basic assumption} holds, and take $K$ as in Theorem \ref{thm:c_B_is_achieved}. Assume that $B$ satisfies
\[
\beta_{ij} \geq 0 \quad \forall (i,j)\in \Kcal_1, \qquad -\infty <\beta_{ij}< K \quad \forall (i,j)\in \Kcal_2,
\]
and for some $l:=a_{\bar h}$ with $\bar h\in \{1,\ldots, m\}$,
\begin{equation}\label{eq:sign assumption}
\begin{split}
&\beta_{ij}>0 \quad \text{for every $(i,j) \in \{1,\dots,l\}^2 \cup \{l+1,\dots,m\}^2$} \\
& \beta_{ij} < 0 \quad \text{for every $(i,j) \in \{1,\dots,l\} \times \{l+1,\dots,m\}$}.
\end{split}
\end{equation}
Then any nonnegative $\mf{u}$ achieving $c = \inf_{\mathcal{N}} J$ is such that $(u_1,\dots,u_l)$ and $(u_{l+1},\dots,u_d)$ are foliated Schwartz symmetric with respect to antipodal directions.
\end{theorem}

The interpretation of the theorem is the following. A $m$-decomposition of $d$ induces a separation of the $d$ components into $m$ different groups: $\{u_i:\ i\in I_h\}$, $h=1,\ldots, m$ (with $I_h$ defined in \eqref{def Ih}).  We join the $m$ groups into two macro groups, the first one collecting the first $\bar h$ groups $\{u_i:\ i\in I_1\cup \ldots \cup I_{\bar h}\}$, the second one the remaining components. If we assume that the relation between components in the same macro-group is purely cooperative, while the relation between components in different macro-groups is purely competitive (see assumption \eqref{eq:sign assumption}), then pairs of components of different macro-groups are foliated Schwartz symmetric with respect to antipodal directions. 

Up to our knowledge, in the literature all the symmetry results so far were for systems in the pure cooperative or pure competitive cases. We point out also that Theorem \ref{thm: partial symmetry} in the particular case $d=2$ and $l=1$ (competition between two components) permits to recover Theorem 1.3 in \cite{TavaresWeth} in the present setting. Moreover, when $m=1$ and $l=d$ (purely cooperative setting), we recover and significantly extend the results in \cite{WangWillem} for $N=2,3$. Further remarks and comments are postponed to Section \ref{sec: partial symmetry}, but we would like to remark here that, in general, least energy positive solutions of \eqref{system} when $\Omega$ is radial are \emph{not} radially symmetric, see Remark 5.4 in \cite{TavaresWeth}, the results of Section 3 in \cite{WangWillem}, and Corollary 0.5 in \cite{SatoWang}. This break of symmetry can be caused  either due to the presence of competition terms, or by the non convexity of the underlying domain.

Clearly one can now combine Theorem \ref{thm: partial symmetry} with the existence of least energy positive solutions of Subsection \ref{sub: existence}. As particular relevant cases, we would like to highlight that we have existence and symmetry of least energy positive solutions whenever $B$ satisfies the following:
\begin{itemize}
\item For a system with $d\geq 2$ equations, if $0<\beta_{ij}<K$ (where $K$ is given by \eqref{eq:special_shape_of_K}), then \eqref{system 2 comp} admits a least energy positive solution, whose components are  foliated Schwartz symmetric with respect to the same point.
\item In the special case $d=3$, and assuming moreover that $\lambda_1=\lambda_2=\lambda$, we have symmetry for least energy solutions $(u_1,u_2,u_3)$ in all the possible situations described in \eqref{bar b_nderbar_b}:
\begin{itemize}
\item if either $\beta_{12}=\beta_{23}=\beta_{13}>\bar \beta$ and $\lambda_3=\lambda$, or $\beta_{12}>\bar \beta$ and $0<\beta_{13}=\beta_{23}<\underline \beta$, then $(u_1,u_2,u_3)$ are foliated Schwartz symmetric with respect to the same point.
\item if $\beta_{12}>\bar \beta$ and $\beta_{13}=\beta_{23}<0$, then $(u_1,u_2)$ and $u_3$ are foliated Schwartz symmetric with respect to antipodal points.
\end{itemize}
\end{itemize}

\subsection{Structure of the paper} In Section \ref{sec: ex non-negative} we prove Theorem \ref{thm:c_B_is_achieved}. Although in the introduction we presented the partial symmetry results at last, we point out that Theorem \ref{thm: partial symmetry} regards not only least energy positive solutions, but constrained minimizers found under the assumptions of Theorem \ref{thm:c_B_is_achieved}. For this reason, the proof of Theorem \ref{thm: partial symmetry} is the object of Section \ref{sec: partial symmetry}. Section \ref{sec: new ex results} is devoted to the proof of the new results on least energy positive solutions: Theorems \ref{thm: weak coop and strong coop 1} and \ref{thm: weak coop and strong coop 2}. Finally, in the last section we prove Theorem \ref{thm: les in space comp}.

\section{Existence of nonnegative minimizers}\label{sec: ex non-negative}

This section is devoted to the proof of Theorem \ref{thm:c_B_is_achieved}. In the literature, minimization on Nehari type sets as $\mathcal{N}$ is usually addressed by firstly studying properties of minimizing sequences, and then showing that any limit of such a sequence can be projected on $\mathcal{N}$. The second part of this argument is extremely delicate from a technical point of view when the number of components is arbitrary (see e.g. Lemmas 2.6-2.8 in \cite{Soave}). In what follows we use a different approach based on the Ekeland's variational principle for the constrained functional $J|_{\mathcal{N}}$. As we shall see, this permits both to avoid several technicalities, and to obtain an \emph{explicit constant} $K$ depending only on $\beta_{ii}$ and $\lambda_i$. Fix $\Omega$ be a bounded domain, $d \ge 2$, and take $\mf{a} = (a_0,a_1,\dots,a_m)$ be a $m$-decomposition of $d$ for some $1 \le m \le d$. We always assume that \eqref{basic assumption} is in force, and we use both the notation introduced in Subsection \ref{sub: existence}, and the following:
\begin{itemize}
\item we let
\begin{equation}\label{def of S}
S:= \inf_{i=1,\dots,d} \inf_{ u \in H_0^1(\Omega) \setminus \{0\} } \frac{ \| u \|_i^2 }{|u|_{L^4}^2}.
\end{equation}
By Sobolev embedding, $S >0$.
\item let $\mf{u} \in H_0^1(\Omega,\R^d)$. We set, for $h=1,\dots,m$,
\begin{equation}\label{def uh}
\mf{u}_h:= \left(u_{a_{h-1}+1},\dots,u_{a_h}\right) \in (H_0^1(\Omega))^{a_h-a_{h-1}}.
\end{equation}
The space $(H_0^1(\Omega))^{a_h-a_{h-1}}$ is naturally endowed with scalar product and norm
\[
\langle \mf{v}^1, \mf{v}^2 \rangle_h := \sum_{i \in I_h} \langle v_{i}^1, v_i^2 \rangle_i \quad \text{and} \quad \|\mf{v}\|_h^2:=\langle \mf{v}, \mf{v} \rangle_h.
\]
\item It will be useful to consider the following set, which contains some weak $H^1_0$ limits of elements of $\Ncal$:
\begin{equation}\label{def tilde N}
\widetilde \Ncal= \left\{ \mf{u} \in H^1_0(\Omega;\R^d) \left|\  
\begin{array}{l}
\|\mf{u}_h\|_h \neq 0 \text{ and } \sum_{i \in I_h} \partial_i J (\mf{u}) u_i \leq 0 \\
\text{for every } h=1,\dots,m
\end{array}\right.\right\}.
\end{equation}
\end{itemize}
Finally, we introduce
\begin{equation}\label{def E}
\mathcal{E}:= \left\{ \mf{u} \in H^1_0(\Omega;\R^d):\  \ M_B(\mf{u})  \text{ is strictly diagonally dominant}  \right\},
\end{equation}
where the $m \times m$ matrix $M_B(\mf{u})$ is defined by 
\[
M_B(\mf{u})_{hk}:= \left( \sum_{(i,j) \in I_h \times I_k} \int_{\Omega} \beta_{ij} u_i^2 u_j^2\right)_{h,k=1,\dots,m}.
\]
Recall that a $m \times m$ matrix $A=(a_{ij})_{i,j}$ is strictly diagonally dominant if for every $i=1,\dots,m$ there holds $a_{ii}> \sum_{j \neq i} |a_{ij}|$. Recall that, if a square matrix $A$ is strictly diagonally dominant and has positive diagonal terms, then it is positive definite. Thus, in particular, for each $\mf{u}\in \mathcal{E}$, $M_B(\mf{u})$ is positive definite. This will be a key property of the matrices $M_B(\mf{u})$. Since we deal with $m\times m$ matrices, with $m$ arbitrary, it does not seem easy to check directly that for some $\mf{u}$ the matrix $M_B(\mf{u})$ is positive definite, and will be always proved by checking that it is strictly diagonally dominant (a condition which involves only the verifications of some inequalities).

Firstly, we recall some basic facts about the geometric structure of $\mathcal{N} \cap \mathcal{E}$, for which we refer to Proposition 1.1 and Remark 9 in \cite{Soave}. The set $\mathcal{E}$ is an open set in $H_0^1(\Omega, \R^d)$. The set $\mathcal{N}$ is defined by a systems of inequalities ($\|\mf{u}_h\|_h >0$ for $h=1,\dots,m$) and a system of equations $G_h(\mf{u}) = 0$, where
\begin{equation}\label{def G_h}
G_h(\mf{u}): = \|\mf{u}\|_h^2 - \sum_{k=1}^m \sum_{(i,j) \in I_h \times I_k} \int_{\Omega} \beta_{ij} u_i^2 u_j^2.  
\end{equation}
It is not difficult to check that $\mathcal{N} \cap \mathcal{E} \neq \emptyset$, and that if $\mf{u} \in \mathcal{N} \cap \mathcal{E}$, then $\mathcal{N}$ is a smooth manifold of codimension $m$ in a neighbourhood of $\mf{u}$. Furthermore, one can show that $\mathcal{N} \cap \mathcal{E}$ is a natural constraint, that is, critical points of $J$ restricted on $\mathcal{N} \cap \mathcal{E}$ are critical points of $J$ in the whole space $H_0^1(\Omega,\R^d)$.

Finally, we recall that with the notation previously introduced the functional $J$ can be written as
\[
J(\mf{u})= \frac{1}{2}\sum_{h=1}^m \|\mf{u}_h\|_h^2 - M_B(\mf{u}) \mf{1} \cdot \mf{1},
\]
where $\mf{1}=(1,\dots,1)$, and $\cdot$ denotes the Euclidean scalar product, and that the constrained functional $J|_{\mathcal{N}}$ reads as 
\begin{equation}\label{funct on N}
J(\mf{u}) = \frac{1}{4} \sum_{i=1}^d \|u_i\|_i^2 = \frac{1}{4} \sum_{i,j=1}^d \int_{\Omega} \beta_{ij} u_i^2 u_j^2,
\end{equation}
so that $J|_{\mathcal{N}}$ is coercive and bounded from below. Thus, it makes sense to search for a constrained minimizer for
\[
c= \inf_{\mathcal{N}} J \ge 0,
\]
where we recall that $\mathcal{N}$ has been defined in \eqref{def N}.
The following is a refinement of \cite[Lemma 2.1]{Soave}, and it is the first step to obtain the constant $K>0$ in Theorem \ref{thm:c_B_is_achieved}.

\begin{lemma}\label{eq:c_B_universalbound}
It holds
\[
c := \inf_{\mathcal{N}} J \leq \bar C,
\]
with
\begin{equation}\label{eq:uniform_upper_estimate}
\bar C = \frac{1}{4} \max_{h=1,\ldots, m}\min_{i\in I_h} \left\{\frac{(1+\lambda_i)^2}{\beta_{ii}}\right\}  \cdot \mathop{\inf_{\Omega \supset \Omega_1,\dots,\Omega_m \text{ open}} }_{\Omega_i \cap \Omega_j = \emptyset,\ (i\neq j)}  \sum_{h=1}^m  S^2(\Omega_h), 
\end{equation}
$S(\Omega)$ being the best Sobolev constant for the embedding $H^1(\Omega)\hookrightarrow L^4(\Omega)$.
\end{lemma}
\begin{proof}
For each $h$, denote $i_h$ the index achieving $\min_{i\in I_h} \left\{\frac{(1+\lambda_i)^2}{\beta_{ii}}\right\}$. Take $\tilde u_{i_1},\ldots,\tilde u_{i_m}\not \equiv 0$ such that $\tilde u_{i_h}\cdot \tilde u_{i_k} \equiv 0$ whenever $h\neq k$, and define $t_h=\|\tilde u_h\|_h/ (\sqrt{\beta_{i_hi_h}}|\tilde u_h|_{L^4}^2)$.  Define $\mf{\bar u}$ such that $\bar u_{i_h}=t_h \tilde u_{i,h}$ for $h=1,\dots,m$, and $\bar u_i=0$ for $i\neq i_1,\ldots, i_m$. It is clear that $\mf{\bar u}_h\not\equiv 0$ for every $h$, and that $\mf{\bar u}\in \Ncal$. Thus, by definition,
\begin{align*}
c&\leq J(\mf{\bar u})=\frac{1}{4}\sum_{i=1}^d \| \bar u_i\|_i^2 =\frac{1}{4}\sum_{h=1}^m t_h^2 \|\tilde u_{i_h}\|_{i_h}^2\\
	&\leq \frac{1}{4}\sum_{h=1}^m  \frac{(1+\lambda_{i_h})^2}{\beta_{i_h i_h}} \frac{\|\tilde u_{h}\|_{1}^4}{|\tilde u_{h}|_{L^4}^4} \leq  \max_{h=1,\ldots, m} \frac{(1+\lambda_{i_h})^2}{\beta_{i_h i_h}} \sum_{h=1}^m\frac{\|\tilde u_{h}\|_{1}^4}{|\tilde u_{h}|_{L^4}^4}
	\qedhere.
\end{align*}
\end{proof}

The following is a key result, both for the existence result of this section, as well as for the symmetry one in the following.

\begin{lemma}\label{lemma:positive_definite}
For $K=S^2/(16\bar C)>0$ (which depends only on $\lambda_i$, $\beta_{ii}$, $i=1,\ldots, d$) we have that, whenever
\[
-\infty < \beta_{ij}< K \qquad \forall (i,j)\in \Kcal_2,
\]
the following inclusion holds:
\[
\widetilde \Ncal \cap \left\{\mf{u}:\ \sum_{i=1}^d \|u_i\|_i^2 \leq 8 \bar C\right\}\subset \mathcal{E},
\]
where we recall that $\widetilde \Ncal$ and $\mathcal{E}$ have been defined in \eqref{def tilde N} and \eqref{def E} respectively.
\end{lemma}

\begin{proof}
Let us prove that $M_B(\mf{u})$ is strictly diagonally dominant, that is, for each $h=1.\ldots, m$,
\begin{equation}\label{eq:equivalent_to_strictdiagdominant}
\sum_{(i,j)\in I_h^2}\int_\Omega \beta_{ij} u_i^2u_j^2 >\mathop{\sum_{k=1}^m}_{k\neq h}\left| \sum_{(i,j)\in I_h\times I_k} \int_\Omega \beta_{ij} u_i^2u_j^2\right|
\end{equation}
Some of the terms inside the absolute value might be positive, while others might be negative. Suppose, without loss of generality, that there exists $\bar m\in \{0,\ldots, m\}$ such that
\[
\mathop{\sum_{k=1}^m}_{k\neq h}\left| \sum_{(i,j)\in I_h\times I_k} \int_\Omega \beta_{ij} u_i^2u_j^2\right|= - \mathop{\sum_{k=1}^{\bar m}}_{k \neq h} \sum_{(i,j)\in I_h\times I_k} \int_\Omega \beta_{ij} u_i^2u_j^2 + \mathop{\sum_{k={\bar m+1}}^{m}}_{k \neq h} \sum_{(i,j)\in I_h\times I_k} \int_\Omega \beta_{ij} u_i^2u_j^2.
\]
On the other hand, since $\mf{u}\in \widetilde \Ncal$, for each $h=1,\ldots, m$,
\[
\sum_{(i,j)\in I_h^2}\int_\Omega \beta_{ij} u_i^2u_j^2\geq \sum_{i\in I_h}\|u_i\|_i^2- \mathop{\sum_{k=1}^m}_{k\neq h} \sum_{(i,j)\in I_h\times I_k} \int_\Omega \beta_{ij}u_i^2 u_j^2.
\]
Thus \eqref{eq:equivalent_to_strictdiagdominant} is true if we show that
\begin{equation}\label{eq:implication_to_strictdiagonaldominant}
 \sum_{i\in I_h}\|u_i\|_i^2 > 2 \mathop{\sum_{k={\bar m+1}}^{m}}_{k \neq h} \sum_{(i,j)\in I_h\times I_k} \int_\Omega \beta_{ij} u_i^2u_j^2.
\end{equation}
which holds if, for $(i,j)\in \Kcal_2$, $\beta_{ij}< K:=S^2/(16 \bar C)$. Indeed, recalling that by assumption $\sum_i \|u_i\|_i^2 \le 8 \bar C$, we have
\begin{align*}
2 \mathop{\sum_{k={\bar m+1}}^{m}}_{k \neq h} \sum_{(i,j)\in I_h\times I_k} \int_\Omega \beta_{ij} u_i^2u_j^2 &< \frac{2K}{S^2} \mathop{\sum_{k={\bar m+1}}^{m}}_{k \neq h} \sum_{(i,j)\in I_h\times I_k} \| u_i\|_i^2 \|u_j\|_j^2\\
					&\leq \frac{16K\bar C}{S^2}  \sum_{i\in I_h} \| u_i\|_i^2 = \sum_{i\in I_h} \| u_i\|_i^2,
\end{align*}
thanks to the choice of $K$.
\end{proof}
 
An immediate consequence is the following.
\begin{lemma}\label{lemma:N_is_manifold}
If 
\[
-\infty<\beta_{ij}< K \qquad \forall (i,j)\in \Kcal_2,
\]
then $\Ncal$ is a manifold at each $\mf{u}\in \Ncal$ with $J(\mf{u})<2\bar C$. Moreover, constrained critical points of $J|_{\mathcal{N}}$ such that $J(\mf{u})<2 \bar C$ are in fact free critical points of $J$. 
\end{lemma}
\begin{proof}
If $\mf{u}\in \Ncal$ and $J(\mf{u})<2\bar C$, then $\mf{u} \in\widetilde{\mathcal{N}}$ and $\sum_i \|u_i\|_i^2 \le 8\bar C$. Thanks to the previous lemma, we deduce that $\mf{u}\in \mathcal{N} \cap \mathcal{E}$, thus $\mathcal{N}$ is a manifold at $\mf{u}$  (see \cite[Remark 9]{Soave}), while the other conclusion comes from \cite[Proposition 1.2]{Soave}.
\end{proof}

Now we can prove that minimizing sequences for $c$ are also conveniently bounded from below.

\begin{lemma}\label{lemma:lower_bounds}
Take $\delta:=S/(2d)>0$. If 
\[
\beta_{ij} \ge 0 \quad \forall (i,j) \in \mathcal{K}_1 \quad \text{and} \quad -\infty<\beta_{ij} < K \quad \forall (i,j) \in \mathcal{K}_2,
\]
then for every $\mf{u}\in \Ncal$ such that $J(\mf{u})\leq 2\bar C$ there holds
\[
\left(\max_{i,j\in I_h} \beta_{ij} \right) \cdot \sum_{i\in I_h} |u_i|^2_{L^4}\geq \delta \qquad \forall h=1,\dots,m.
\]
\end{lemma}
\begin{proof}
Following the proof of Lemma 2.2 in \cite{Soave}, since $\mf{u}\in \Ncal$:
\begin{align*}
S \sum_{i\in I_h} |u_i|_{L^4}^2 &\leq \sum_{i\in I_h} \|u_i\|_i^2 =\sum_{(i,j)\in I_h^2} \int_\Omega \frac{\beta_{ij}}{2}\left(u_i^4+u_j^4\right) + K \sum_{i\in I_h} \sum_{j\not\in I_h} |u_i|_{L^4}^2|u_j|_{L^4}^2\\
						&<\max_{i,j\in I_h}\{\beta_{ij}\} \sum_{i\in I_h}|u_i|_{L^4}^4+ \frac{8KC}{S}\sum_{i\in I_h}|u_i|_{L^4}^2\\
						&<d \max_{i,j\in I_h}\{\beta_{ij}\} \left(\sum_{i\in I_h}|u_i|_{L^4}^2\right)^2+ \frac{S}{2}\sum_{i\in I_h}|u_i|_{L^4}^2. \qedhere
\end{align*}
\end{proof}

Having established the basic properties of minimizing sequences, we can proceed with the core of the argument.

\begin{lemma}
The constrained functional $J|_{\Ncal}$ satisfies the Palais-Smale condition at level $c$, whenever $B$ is such that
\[
\beta_{ij}\geq 0 \quad \forall (i,j)\in \Kcal_1 \quad \text{and} \quad -\infty< \beta_{ij}< K \quad \forall (i,j)\in \Kcal_2.
\]
\end{lemma}
\begin{proof}
Take $\{\mf{u}_n\} \subset \mathcal{N}$ such that
\begin{equation}\label{constrained PS sequence}
J(\mf{u}_n)\to c,\qquad J'(\mf{u}_n)=\sum_{h=1}^h\lambda_{h,n}G'_h(\mf{u}_n)+\textrm{o}(1),
\end{equation}
where we recall that $G_h$ has been defined in \eqref{def G_h}. We can take $n$ large enough so that $J(\mf{u}_n)\leq 2\bar C$.
By \eqref{funct on N}, up to a subsequence
\[
u_{i,n}\to u_i \qquad \text{ weakly in $H^1_0(\Omega)$, strongly in $L^2(\Omega)\cap L^4(\Omega)$}, 
\]
whence
\begin{itemize}
\item[($i$)] $\mf{u}\in \widetilde \Ncal$ and $\sum_{i} \|u_{i}\|_{i}^2 \le 8 \bar C$;
\item[($ii$)] $M_B(\mf{u}_n) \to M_B(\mf{u})$ component--wise.
\end{itemize}
From ($i$), by Lemma \ref{lemma:positive_definite} we deduce that $\mf{u} \in \mathcal{E}$, so that $M_B(\mf{u})$ is positive definite. Testing the second equation in \eqref{constrained PS sequence} with $\hat{\mf{u}}^h \in H_0^1(\Omega,\R^d)$ defined by 
\[
\hat{u}_{i,n}^h:= \begin{cases} u_{i,n} & \text{if $i \in I_h$} \\
0 & \text{if $i \not \in I_h$},
\end{cases}
\]
we obtain with ($ii$) that
\[
\textrm{\bf o}(1)=M_B(\mf{u}_n) \left(\begin{array}{c} \lambda_{1,n} \\ \vdots \\ \lambda_{m,n} \end{array}\right)=(M_B(\mf{u})+\textrm{\bf o}(1))\left(\begin{array}{c} \lambda_{1,n} \\ \vdots \\ \lambda_{m,n} \end{array}\right);
\]
multiplying by $(\lambda_{1,n},\ldots, \lambda_{m,n})$, and using the fact that $M_B(\mf{u})$ is positive definite, we finally infer
\[
\textrm{\bf o}(1) |(\lambda_{m,n},\ldots, \lambda_{m,n})|\geq C |(\lambda_{1,n},\ldots, \lambda_{m, n})|^2+ \textrm{\bf o}(1) |(\lambda_{1,n},\ldots,\lambda_{m,n})|^2,
\]
yielding $\lambda_{i,n}\to 0$. Moreover, $\{G_h'(\mf{u}_n)\}$ is a uniformly bounded family of operators, thanks to the boundedness of $\{\mf{u}_n\}$, and hence $J'(\mf{u}_n)\to 0$. This means that $\mf{u}_n$ is a standard Palais-Smale sequence, and the result follows easily from now on.
\end{proof}

We are now ready to prove the main result of this section.
 
\begin{proof}[Proof of Theorem \ref{thm:c_B_is_achieved}]
The proof is a simple consequence of what was established before. In fact, $c\geq 0$, hence we can take a minimizing sequence $\mf{u}_n$, which we can choose, by Ekeland's variational principle, to be a Palais-Smale sequence for $J|_{\Ncal}$ at level $c$. Note that the Ekeland's principle is applicable, since by Lemma \ref{lemma:lower_bounds} the set $\mathcal{N} \cap \{J \le  2 \bar C\}$ endowed with the $H^1_0(\Omega)$ topology is a complete metric space. Thus, by the previous lemma, up to a subsequence $\mf{u}_n\to \mf{u}$ strongly in $H^1_0(\Omega)$, and by Lemma \ref{lemma:lower_bounds} this implies that $\mf{u}\in \Ncal$. By convergence, we infer that
\[
J(\mf{u}) = \lim_{n \to \infty} J(\mf{u}_n) = c,
\]
which completes the proof since $\Ncal \cap\{J<2\bar C\}$ is a natural constraint (cf. Lemma \ref{lemma:N_is_manifold}).  
\end{proof}

\begin{remark}\label{rem: constant K}
For future reference, we observe that the constant $K$ is equal to $S^2/(16\bar C)$, where 
\[
S:= \inf_{i=1,\dots,d} \inf_{ u \in H_0^1(\Omega) \setminus \{0\} } \frac{ \| u \|_i^2 }{|u|_{L^4}^2},
\]
and
\[
\bar C = \frac{1}{4} \max_{h=1,\ldots, m}\min_{i\in I_h} \left\{\frac{(1+\lambda_i)^2}{\beta_{ii}}\right\}  \cdot \mathop{\inf_{\Omega \supset \Omega_1,\dots,\Omega_m \text{ open}} }_{\Omega_i \cap \Omega_j = \emptyset,\ (i\neq j)}  \sum_{h=1}^m  S^2(\Omega_h)
\]
(recall Lemma \ref{lemma:positive_definite} and \eqref{eq:uniform_upper_estimate}).
\end{remark}

\begin{remark}
In the particular case of full cooperative systems, we can have a better (and more explicit) constant $K>0$. Thinking for instance at the $d$--decomposition $\textbf{a}=(0,1,\ldots, m)$, one can take:
\begin{equation}\label{eq:a better K}
K:= \frac{\min_{i=1,\ldots, d} \{S_i^2\}}{2\sum_{j=1}^m \frac{S_j^2}{\beta_{jj}}},\quad  \text{ where } S_i:=\inf_{\int_\Omega u^4=1} \|u\|_i^2.
\end{equation}
This constant is similar to the one appearing in assumption (H2) of \cite{WangWillem} (for $d=2$), being ours slightly worse in the framework of Z.-Q. Wang and M. Willem's paper. This is a price to pay from passing from $d=2$ to more equations, since we had to prove that $M_B(\mf{u})$ is positive definite by proving that actually it is strictly diagonally dominant (while in $2\times 2$ matrices one can perform an explicit computation). 
The proof of \eqref{eq:a better K} is not completely immediate, but since we consider this to be a lateral statement, here we just provide some hints. By taking $\tilde w_i$ to be functions achieving $S_i$, we can take $w_i=\sqrt{S_i/\beta_{ii}} \tilde w_i$, which satisfies the equation $-\Delta w_i+\lambda_i w_i=\beta_{ii} w_i^3$, with $\|w_i\|_i^2=S_i^2/\beta_{ii}$. Moreover, since we are in a full cooperative case, it is straightforward to check that $\mf{w}\in \widetilde \Ncal$. By using \eqref{eq:a better K}, $\mf{w}\in \Ecal$ and there exist $0<t_i<1$ such that $(t_1 w_1,\ldots, t_d w_d) \in \Ncal$. Thus $c\leq \sum_{j=1}^d S_j^2/(4\beta_{jj})$, and working through the proof of Lemma \ref{lemma:positive_definite} the rest follows.
\end{remark}

\section{Partial symmetry of nonnegative minimizers}\label{sec: partial symmetry}

We now turn to the problem of the symmetry of least energy positive solutions when \emph{$\Omega$ is a radially symmetric bounded domain of $\R^N$} (which we always assume along this section). Observe that when $\Omega$ is a ball, the classical result by Troy \cite{Troy} (see also \cite{Shaker}) yields that, in the cooperative case, each positive solution is radially symmetric. However, thanks to \cite[Theorem 1.4]{TavaresWeth} (which deal with the competitive case), \cite[Corollary 0.5]{SatoWang} (mixed cooperation and competition case), and \cite[Section 3]{WangWillem} (cooperative case, $\Omega$ an annulus), it is known that \emph{in general} a least energy positive solution is not radial. 
On the other hand, it is natural to expect that least energy solutions inherit part of the symmetric structure of the problem. Theorem \ref{thm: partial symmetry} establishes that this is the case when the competition takes place between two groups of cooperative components, or in a purely cooperative setting. 

\subsection{Comments on Theorem \ref{thm: partial symmetry}}

Due to its general formulation, Theorem \ref{thm: partial symmetry} might not be easy to read and to understand. In this subsection we present several remarks which should help the reader towards this purpose.

\begin{remark} 
The reader could object that only the division in macro-groups is necessary, and that one could drop the original division into $m$ groups considering only a $2$ decomposition of $d$ in $2$ different groups. This is a a particular case of our result, but it is not equivalent, because in our statement we allow to minimize in different Nehari type sets (related to the division in $m$ groups), and so in this way we can deal with a larger class of constrained minimizers.

As an illustrative example, we consider a $3$ components system, separating $(u_1,u_2,u_3)$ into $2$ macro-groups $(u_1,u_2)$ and $u_3$. This can be the result of two different decompositions:
\begin{itemize}
\item let us consider the natural $2$-decomposition $(0,2,3)$. By Theorem 1.7 in \cite{Soave}, it is know that if $\beta_{12}>C\max\{\beta_{11},\beta_{22}\}$ for a positive constant $C>0$ depending on $\lambda_i$, and $\beta_{13}, \beta_{23} \ll -1$, then the minimum of $J$ on $\mathcal{N}_{(0,2,3)}$ is achieved by a \emph{positive} solution of \eqref{system} (here $\mathcal{N}_{(0,2,3)}$ is the Nehari set determined by the $2$-decomposition of $3$). The same result holds true if $\beta_{13},\beta_{23} < 0$ and $\beta_{12} \gg 1$, as proved in Theorem 0.1 in \cite{SatoWang}. In both cases, Theorem \ref{thm: partial symmetry} applies proving that $(u_1,u_2)$ and $u_3$ are foliated Schwartz symmetric with respect to antipodal points. 
\item Let us now consider the $3$-decomposition $(0,1,2,3)$, denoting by $\mathcal{N}_{(0,1,2,3)}$ the corresponding Nehari set. By Corollary \ref{corol: comp e weak coop}, there exists $K>0$ such that if $\beta_{ij} < K$ for every $i \neq j$, then the minimum of $J$ on $\mathcal{N}_{(0,1,2,3)}$ is achieved by a \emph{positive} solution of \eqref{system}. If we assume further that $0<\beta_{12}<K$ and $\beta_{13},\beta_{23} <0$, then by Theorem \ref{thm: partial symmetry} we obtain that $(u_1,u_2)$ and $u_3$ are foliated Schwartz symmetric with respect to antipodal points. 
\end{itemize}
The second result would not have been obtained if we had considered only the $2$-decomposition $(0,2,3)$ in our symmetry result. 
\end{remark}

Let us now make some comments regarding Theorem \ref{thm: partial symmetry} in the \emph{purely cooperative case}, which correspond to $l=d$ in the assumptions. 

\begin{remark}\label{rem:WangWillem}
In \cite{WangWillem}, Z.-Q. Wang and M. Willem proved partial symmetry results in two situations. For a system with $d$ components, they showed that if the infimum of $J$ on the natural Nehari manifold
\[
\left\{ \mf{u} \in \mathbb{H}\, \middle| \, \mf{u} \neq \mf{0} \text{ and } \sum_{i=1}^d \pa_i J(\mf{u}) u_i = 0\right\}
\]
is achieved (which is suitable only for large cooperation rates), then any positive minimizer is such that all the components $(u_1,\dots,u_d)$ are foliated Schwartz symmetric with respect to the same point. Moreover, for systems of $d=2$ components, they showed that if $0<\beta_{12}<C(\beta_{11},\beta_{22},\lambda_1,\lambda_2)$, then the infimum of $J$ on the Nehari set
\[
\left\{ \mf{u} \in \mathbb{H}\, \middle|\, \begin{array}{l} u_i \neq \mf{0} \text{ and } \pa_i J(\mf{u}) u_i = 0\\ 
\text{for $i=1,2$}\end{array}\right\}
\]
is achieved by a positive solution of \eqref{system}, and any minimizer is such that $u_1,u_2$ are foliated Schwartz symmetric with respect to the same point. The restriction $d=2$ is relevant in their proof. 
\end{remark}

Our result recovers both the ones in \cite{WangWillem}, extend the second one to systems with an arbitrary number of components, and provide partial symmetry also when $m$-decompositions of $d$ with $m \neq 1,d$ are considered. For this reason, we think that Theorem \ref{thm: partial symmetry} in the purely cooperative case deserves a statement on its own.

\begin{corollary}\label{thm: partial sym cooperative}
Let $d \ge 2$, and let $\mf{a}$ be a $m$-decomposition of $d$. Take $K>0$ be as in Theorem \ref{thm:c_B_is_achieved}. If
\[
\beta_{ij} > 0 \quad \forall (i,j)\in \Kcal_1, \qquad 0<\beta_{ij}< K \quad \forall (i,j)\in \Kcal_2,
\]
then any nonnegative minimizer $\tilde{\mf{u}}$ of $J$ constrained on $\mathcal{N}$ is such that all the components $\tilde u_1,\dots,\tilde u_d$ are foliated Schwartz symmetric with respect to the same point.
\end{corollary}

\begin{remark}\label{rem: properties minimizers}
Concerning the proof of Theorem \ref{thm: partial symmetry}, in the literature the partial symmetry of solutions of elliptic equations is often obtained through an $\inf \sup$ characterization. Dealing with systems, it turns out to be very complicated to obtain such a variational characterization; for instance, this is why the proof of Theorem 1.2 in  \cite{WangWillem} works only for systems with $2$ components. Thus, we think that it is worth to point out that we will not use any $\inf \sup$ characterization, basing our argument directly on the constrained minimality.
\end{remark}

\subsection{Proof of Theorem \ref{thm: partial symmetry}}
In what follows, without loss of generality, we suppose that $\Omega$ is radial with respect to $0$. We will use polarization techniques, and hence at first we recall some definitions which are by now classic.

Assume $H$ is a closed half-space in $\R^N$. We denote by  $\sigma_{H}: \R^N \to \R^N$  the reflection with respect to the boundary $\partial H$ of $H$. For a measurable function $w : \R^N \to \R$ we define the polarization $w_H$ of $w$ relative to $H$ by
\[
w_H(x) =
\left\{
\begin{array}{l}
\max\{ w(x), w(\sigma_H(x)) \}, \quad x \in H,\\
\min\{ w(x), w(\sigma_H(x))  \}, \quad x \in \R^N \setminus H.
\end{array}
\right.
\]

We consider the set $\mathcal{H}_0$ of all closed half-spaces $H$ in $\R^N$ such that $0 \in \partial H$. Given an unitary vector $p \in \S^{N-1}$, we denote by $\mathcal{H}_{0}(p)$ the set of all closed half-spaces $H \in \mathcal{H}_0$ such that $p \in \textrm{int}(H)$.

Recall that a function $f:\R^N\to \R$ is said to be foliated Schwarz symmetric with respect to a unitary vector $p$ if it is axially symmetric with respect to the axis $\R p$ and nonincreasing in the polar angle $\theta= \arccos( p \cdot  x/|x|) \in [0, \pi]$.  We mention that, up to our knowledge, the link between polarization and foliated Schwarz symmetry appeared firstly in \cite{SmetsWillem, Brock}. We would like to mention also the precursory works \cite{Ahlfors,Baernstein2,BrockSolynin} that brought to light the relation between polarizations and rearrangements in many different settings, and refer to the survey \cite{Weth} for a  detailed history of the subject. The following is a useful alternative characterization of foliated Schwarz symmetry, and we refer to \cite[Lemma 4.2]{Brock} (see also \cite[Proposition 2.7]{Weth}) for the proof.

\begin{lemma}\label{lemma: equivalent charact for Sch symmetry}
Let $\Omega$ a radial set centered at the origin, and let $u:\Omega\to \R$ be a continuous function. Then $u$ is foliated Schwarz symmetric with respect to $p\in \mathbb{S}^{N-1}$ if, and only if, for every $H\in \Hcal_0(p)$ we have $u(x)\geq u(\sigma_H(x))$ whenever $x\in \Omega\cap H$.
\end{lemma}

For every $H\in \Hcal$ we denote by $\widehat H\in \Hcal_0$ the closure
of the complementary half-space $\R^N\setminus H$. In the spirit of \cite{BartschWethWillem,TavaresWeth,WangWillem}, the proof of Theorem \ref{thm: partial symmetry} is based upon a general criterion (cf. for instance Theorem 2.6 in \cite{BartschWethWillem} or Theorem 4.3 in \cite{TavaresWeth}). 

\begin{proposition}\label{prop: abstract criterion}
Let $\mf{u}$ be a nonnegative solution of \eqref{system}. If for every $H \in \mathcal{H}_0$ the function 
\[
\mf{u}^H:=(\mf{u}_{1,H}, \dots, \mf{u}_{\bar h,H},\mf{u}_{\bar h+1, \widehat{H}},  \dots, \mf{u}_{m, \widehat{H}}) = (u_{1,H}, \dots, u_{l,H},u_{l+1,\widehat{H}}, \dots, u_{d, \widehat{H}})
\]
is still a solution of \eqref{system}, then $(u_1,\dots,u_l)$ and $(u_{l +1},\dots,u_d)$ are foliated Schwartz symmetric with respect to antipodal points. 
\end{proposition}

\begin{proof}
Since $\mf{u} \in \mathcal{N}$, there exists an index $i \in I_1$ such that $u_i \not \equiv 0$ in $\Omega$. Without loss of generality, we assume $i=1$. Let $r>0$ be such that $\pa B_r(0) \subset \Omega$, and let $p \in \mathbb{S}^{N-1}$ be such that $\max_{\pa B_r(0)} u_1 = u_1(rp)$. By assumption, for any $H \in \mathcal{H}_0(p)$, 
\[
-\Delta u_{1,H}+\lambda_1 u_{1,H} = \sum_{j=1}^l \beta_{1j} u_{1,H} u_{j,H}^2 + \sum_{j=l+1}^d \beta_{1j} u_{1,H} u_{j,\widehat{H}}^2 \qquad \text{in $\Omega$}.
\]
Thus, if we let $w:= u_{1,H}-u_1$, we obtain
\begin{align*}
-\Delta w + \lambda_1 w & = \sum_{j=1}^l \beta_{1j} \left( u_{1,H} u_{j,H}^2 - u_1 u_j^2 \right)  + \sum_{j=l+1}^d \beta_{1j} \left( u_{1,H} u_{j,\widehat{H}}^2-u_1 u_j^2\right) \\
& = \sum_{j=1}^l \beta_{1j} u_{1,H}\left(u_{j,H}^2-u_j^2\right) + \sum_{j=l+1}^d \beta_{1j} u_{1,H} \left(u_{j,\widehat{H}}^2-u_j^2\right) +\sum_{j=1}^d \beta_{1j} u_j^2 (u_{1,H}-u_1),
\end{align*}
in $\Omega \cap H$, that is
\begin{equation}\label{eq22TaWe}
-\Delta w + \underbrace{\left( \lambda_1 -\sum_{j=1}^d \beta_{1j} u_j^2\right)}_{=:q(x) \in L^\infty_{\loc}(\Omega \cap H)} w = \sum_{j=1}^l \beta_{1j} u_{1,H}\left(u_{j,H}^2-u_j^2\right) + \sum_{j=l+1}^d \beta_{1j} u_{1,H} \left(u_{j,\widehat{H}}^2-u_j^2\right) .
\end{equation}
By definition $u_{j,H} \ge u_j$ in $\Omega \cap H$ for every $j=1,\dots,l$, while $u_{j,\widehat{H}} \le u_j$ in $\Omega \cap H$ for every $j=l+1,\dots,d$. Therefore, recalling that $\beta_{1j}>0$ for $j=1,\dots,l$, while $\beta_{1j}<0$ for $j=l+1,\dots,d$, we deduce that $-\Delta w+q(x) w \ge 0$ and $w \ge 0$ in $\Omega \cap H$. The strong maximum principle leads to the alternative $w>0$ in $\Omega \cap H$, or $w \equiv 0$ in $\Omega \cap H$. As $w(rp)=0$, the latter condition holds true, and coming back to equation \eqref{eq22TaWe} we deduce that
\[
0 = -\Delta w + q(x) w = \sum_{j=1}^l \beta_{1j} u_{1,H}\left(u_{j,H}^2-u_j^2\right) + \sum_{j=l+1}^d \beta_{1j} u_{1,H} \left(u_{j,\widehat{H}}^2-u_j^2\right) \ge 0,
\]
which in turn implies (since $u_1 \not \equiv 0 \Rightarrow u_1 > 0$ in $\Omega$) that $u_{j,H} \equiv u_j$ in $\Omega \cap H$ for every $j=1,\dots,l$ and $u_{j,\widehat{H}} \equiv u_j$ in $\Omega \cap H$ for every $j=l+1,\dots,k$. Since $H \in \mathcal{H}_0(p)$ has been arbitrarily chosen, the thesis follows from Lemma \ref{lemma: equivalent charact for Sch symmetry}.
\end{proof}

In the following we will prove that, under the assumptions of Theorem \ref{thm: partial symmetry}, it is possible to apply Proposition \ref{prop: abstract criterion}. First we need some preliminary results.

\begin{lemma}\label{lem: estimate interaction}
For every $u,v \in H^1_0(\Omega)$ and $H \in \mathcal{H}_0(p)$, then also $u_H,v_H\in H^1_0(\Omega)$, and moreover:
\begin{itemize}
\item[(i)] $\displaystyle \int_\Omega |u|^p\, dx=\int_\Omega |u_H|^p$, for every $p\geq 1$.
\item[(ii)]  $\displaystyle \int_\Omega |\nabla u_H|^2=\int_\Omega |\nabla u|^2$.
  \item[(iii)] $\displaystyle \int_{\Omega} u^2 v^2 \le \int_{\Omega} u_{H}^2 v_{H}^2 $ and $\displaystyle \int_{\Omega} u^2 v^2 \ge \int_{\Omega} u_{H}^2 v_{\widehat{H}}^2$.
\end{itemize}
\end{lemma}

\begin{proof}
For the first two items, check for instance \cite[Lemma 3.1]{Weth}). As for (iii), the first inequality follows by Proposition 31.7 in \cite{WillemBook}, while the second one is a particular case of Lemma 4.5 in \cite{TavaresWeth}.
\end{proof}

Assume that $B$ is as in the assumption of Theorem \ref{thm: partial symmetry}, let $c= \inf_{\mathcal{N}} J$, and let $\mf{u}$ be a nonnegative minimizer of $J$ on $\mathcal{N}$. In view of Proposition \ref{prop: abstract criterion}, we aim at proving that $\mf{u}^H:=(u_{1,H},\dots,u_{l,H},u_{l+1,\widehat{H}},\dots,u_{d,\widehat{H}})$ also achieves $c$ for every $H \in \mathcal{H}_0$. In this perspective, the main difficulty consists in showing that $\mf{u}^H\in \mathcal{N}$ (in the literature, this is usually the part of the proof which requires an $\inf \sup$ characterization, see Remark \ref{rem: properties minimizers}). To this aim, we study the function of real variables:
\[
\Psi: (t_1,\dots,t_m) \in \overline{\R_+^m} \mapsto J\left( \sqrt{t_1} \mf{u}_{1,H},\dots, \sqrt{t_{\bar h}} \mf{u}_{\bar h, H}, \sqrt{t_{\bar h+1}} \mf{u}_{\bar h+1, \widehat{H}},\dots, \sqrt{t_m}\mf{u}_{m,\widehat{H}} \right) \in \R,
\]
and claim that, under the considered assumption, it has a unique maximum point $\tilde{\mf{t}}=(\tilde t_1,\ldots, \tilde t_m)$ such that $\tilde t_h >0$ for every $h=1,\ldots, m$. 

By the properties of polarization stated in Lemma \ref{lem: estimate interaction}, and the assumptions made on $B$, we have that
\[
\mf{u}^H\in \widetilde \Ncal,\qquad \sum_{h=1}^m \|\mf{u}_h^H\|^2_h=\sum_{j=1}^d \|u_i\|_i^2=4c\leq 4\bar C
\]
where $\bar C$ is as in \eqref{eq:uniform_upper_estimate}. Hence, thanks to Lemma \ref{lemma:positive_definite}, the matrix $M(\mf{u}^H)$ is positive definite, and observing that 
\[
\Psi(t_1,\dots,t_m) = \frac{1}{2} \sum_{h=1}^m \|\mf{u}^H\|_h^2 t_h - \frac{1}{4} M_B(\mf{u}^H) \mf{t} \cdot \mf{t},
\]
this implies that there exists $\kappa>0$ such that 
\[
\Psi(t_1,\ldots, t_m)\leq \frac{1}{2}\sum_{h=1}^m t_h\|\tilde{\mf{u}}_h^H\|_h^2-\kappa \sum_{h=1}^m t_h^2\to -\infty \qquad \text{ as } |\mf{t}|\to \infty.
\]
Therefore, $\Psi$ admits a global maximum $\tilde{\mf{t}}$ in $\overline{\R_+^m}$. Note also that it is a strictly concave function. Since $\Psi$ is of class $\mathcal{C}^1$ up to the boundary of $\R^m_+$, the maximality of $\tilde{t}$ entails $\pa_h \Psi(\tilde{\mf{t}}) \le 0$ if $\tilde t_h=0$, and $\pa_h \Psi(\tilde{\mf{t}}) = 0$ if $\tilde t_h>0$. In particular, we observe that 
\begin{equation}\label{eq:system_for_the_t}
\|\mf{u}_h^H\|_h^2 = \sum_{h=1}^m M_B(\mf{u}^H)_{hk} \tilde t_k^H \qquad \text{whenever $\tilde t_h>0$}.
\end{equation}

\begin{remark}\label{rem: unique max}
Let us consider the function
\[
\Phi:(t_1,\dots,t_m) \in \R_+^m \mapsto J\left(\sqrt{t_1} \mf{u}_1,\dots, \sqrt{t_m} \mf{u}_m\right).
\]
Under our assumption, by Lemma \ref{lemma:positive_definite} any minimizer for $c$ stays in $\mathcal{E}$. Therefore, as in the previous discussion, we can check that $\Phi$ is strictly concave and has a maximum point in $\overline{\R_+^m}$. Moreover, since $\mf{u} \in \mathcal{N}$, the point $\mf{1}=(1,\dots,1)$ is a critical point for $\Phi$. Hence, by strict concavity, $\mf{1}$ is the unique critical point of $\Phi$, and is a global maximum.
\end{remark}

\begin{lemma}
We have
\[
\sum_{h=1}^m \tilde t_h \|\mf{u}_h^H\|_h^2\leq 4 \bar C.
\]
\end{lemma}
\begin{proof}
We claim that 
\[
\|\mf{u}^H\|_h^2 \tilde t_h = \sum_{k=1}^m M_B(\mf{u}^H)_{hk} \tilde t_h \tilde t_k \qquad \forall h=1,\dots,m.
\]
Indeed, if $\tilde t_h=0$ this relation is trivially satisfied. If $\tilde t_h>0$, then it follows by \eqref{eq:system_for_the_t}. Therefore
\[
J\left( \sqrt{\tilde t_1} \mf{u}_1^H,\dots,\sqrt{\tilde t_m} \mf{u}_m^H \right) = \frac{1}{4} \sum_{h=1}^m \|\mf{u}_h^H\|_h^2 \tilde t_h.
\]
On the other hand, combining Lemma \ref{lem: estimate interaction} with the assumptions on $B$, we have also
\begin{align*}
J\left(\sqrt{\tilde t_1} \mf{u}^H_1,\ldots, \sqrt{\tilde t_m} \mf{u}^H_m\right) &\leq J\left(\sqrt{\tilde t_1} \mf{u}_1,
\ldots, \sqrt{\tilde t_m} \mf{u}_m\right)\leq \sup_{t_1,\ldots, t_m\geq 0} J\left(\sqrt{t_1} \mf{u}_1,\ldots, \sqrt{t_m} \mf{u}_m\right)\\
&=J(\mf{u}_1,\ldots, \mf{u}_m)=c\leq \bar C.
\end{align*}
Notice that we have used the fact that $\Phi$ has the unique maximizer $(1,\dots1)$, see the previous remark.
\end{proof}

\begin{lemma}
There holds
\[
\tilde t_1,\dots, \tilde t_h> 0.
\]
\end{lemma}
\begin{proof}
Assume, in view of a contradiction, that $\tilde t_1=0$. We will check that
\[
J\left(\sqrt{ t_1} \mf{u}^H_1,\ldots, \sqrt{\tilde t_m} \mf{u}^H_m\right) >J\left(0,\sqrt{ \tilde t_2} \mf{u}^H_2,\ldots, \sqrt{\tilde t_m} \mf{u}^H_m\right)
\]
for $t_1>0$ sufficiently close to $0$, which contradicts the maximality of $(0,\tilde t_2,\ldots, \tilde t_m)$. The left hand side of the inequality can be rewritten as
\[
\frac{1}{2}t_1 \|\mf{u}^H_1\|_1^2-\frac{1}{2}\sum_{h=2}^m M_B(\mf{u}^H)_{1h} t_1\tilde t_h -\frac{1}{4}M_B(\mf{u}^H)_{11}t_1^2 + \frac{1}{2}\sum_{h=2}^m \tilde t_h\|\mf{u}^H_h\|_h^2 -\frac{1}{4} \sum_{h,k= 2}^m M_B(\mf{u}^H)_{hk}\tilde t_h \tilde t_k.
\]
Observe that, for $\beta_{ij}< K$ in $\Kcal_2$ (and $K$ as in Remark \ref{rem: constant K}), it results 
\begin{align}
\sum_{h=2}^m M_B(\mf{u}^H)_{1h} \tilde t_h &=\sum_{h= 2}^m\sum_{(i,j)\in I_1\times I_h} \int_\Omega \beta_{ij} ( u_i^H  u^H_j)^2 \tilde t_h \leq \frac{K}{S^2} \sum_{h=2}^m \sum_{(i,j)\in I_1\times I_h} \|u_i^H\|_i^2 \| u_j^H\|_j^2 \tilde t_h\\
										   &\leq \frac{K}{S^2} \|\mf{u}_1^H\|_1^2 \sum_{h=2}^m \tilde t_h \|\mf{u}^H_h\|_h^2\leq \frac{4K \bar C}{S^2}\|\mf{u}_1^H\|_1^2\leq \frac{1}{4}\|\mf{u}_1^H\|_1^2,
\end{align}
where we used the estimate of the previous lemma. Thus
\begin{multline*}
\frac{1}{2}t_1 \|\mf{u}^H_1\|_1^2-\frac{1}{2}\sum_{h=2}^m M_B(\mf{u}^H)_{1h} t_1\tilde t_h -\frac{1}{4}M_B(\mf{u}^H)_{11}t_1^2  \\
=\frac{1}{2}t_1\left( \|\mf{u}^H_1\|_1^2-\sum_{h=2}^m M_B(\mf{u}^H)_{1h} \tilde t_h -\frac{1}{2}M_B(\mf{u}^H)_{11} t_1\right) \ge  \frac{1}{2} t_1 \left(\frac{3}{4}\|\mf{u}^H_1\|_1^2 - t_1 \mf{M}(\mf{u}^H)_{11}\right)>0
\end{multline*}
for sufficiently small small $t_1>0$, which yields the desired contradiction.
\end{proof}

\begin{proof}[End of the proof of Theorem \ref{thm: partial symmetry}] Given $\mf{u}$ achieving $c$ and $H\in \Hcal_0$, we have concluded that there exists a maximizer $\tilde t$ for the function $\Psi$ in $\overline{\R_+^m}$, and $\tilde t_1,\ldots, \tilde t_m>0$. By \eqref{eq:system_for_the_t}, we infer that $(\sqrt{\tilde t_1} \mf{u}^H_1,\ldots, \sqrt{\tilde t_m} \mf{u}^H_m)\in \Ncal$; together with Lemma \ref{lem: estimate interaction}, this implies that
\begin{align}
c & \leq J\left(\sqrt{\tilde t_1} \mf{u}^H_1,\ldots, \sqrt{\tilde t_m} \mf{u}^H_m\right) \leq  J\left(\sqrt{\tilde t_1} \mf{u}_1,\ldots, \sqrt{\tilde t_m} \mf{u}_m\right) \\
& \leq \sup_{t_1,\ldots, t_m\geq 0}  J\left(\sqrt{t_1} \mf{u}_1,\ldots, \sqrt{t_m} \mf{u}_m\right) =J(\mf{u})=c,
\end{align}
which is then a chain of equalities. The uniqueness of the maximum for the function $\Phi$ (see Remark \ref{rem: unique max}) entails $\tilde t_h=1$ for every $h$, and thus $\mf{u}^H$ also achieves $c$, being in particular a solution of \eqref{system} (cf. Lemma \ref{lemma:N_is_manifold}). We can now conclude using the criterion of Proposition \ref{prop: abstract criterion}.
\end{proof}

\section{Existence of least energy positive solutions}\label{sec: new ex results}

This section is devoted to the proofs of Theorems \ref{thm: weak coop and strong coop 1} and \ref{thm: weak coop and strong coop 2}. They are inspired by those of Theorem 1.6 and 1.7 in \cite{Soave}.

\subsection{Proof of Theorem \ref{thm: weak coop and strong coop 1}}

Under the considered assumptions, by Theorem \ref{thm:c_B_is_achieved} there exists a nonnegative solution $\mf{u}$ of \eqref{system} which minimizes $J$ in the Nehari set $\mathcal{N}$. We wish to show that 
\[
\inf_{\mathcal{N}} J < \inf_{\mathcal{N} \cap \{w_i=0 \ \text{for some $i$}\}} J.
\] 
If this is true, then by minimality $u_i \neq 0$ for every $i$. By contradiction, let us assume that for some index $l$ there holds $u_l = 0$. Let $\bar h \in \{1,\dots,m\}$ be such that $l \in I_{\bar h}$. By definition of $\mathcal{N}$, there exists $p \in I_{\bar h}$ such that $u_p \neq 0$. By Lemma \ref{lemma:positive_definite} we know that $\mf{u} \in \mathcal{E}$, and hence $\mathcal{N}$ defines, in a neighbourhood of $\mf{u}$, a smooth manifold (actually a $\mathcal{C}^2$-manifold, as it is immediate to verify) of codimension $m$ in $H_0^1(\Omega;\R^d)$. We claim that 
\begin{equation}\label{claim on second diff}
d^2 J(\mf{u})[\mf{v},\mf{v}] \ge 0 \quad \text{for every $\mf{v} \in T_\mf{u}(\Ncal)$},
\end{equation}
where $T_\mf{u}(\Ncal)$ denotes the tangent space to $\mathcal{N}$ at the point $\mf{u}$.  To prove this, we observe that since $\mathcal{N}$ is of class $\mathcal{C}^2$, for any $\mf{v} \in T_\mf{u}(\Ncal)$ there exists a $\mathcal{C}^2$ curve $\gamma: (-\eps,\eps) \to \mathcal{N}$ for some $\eps >0$ such that $\gamma(0) = \mf{u}$ and $\gamma'(0) = \mf{v}$. Now by minimality of $\mf{u}$, and recalling that $dJ(\mf{u})= 0$, we infer that
\[
0 \le \left. \frac{d^2}{dt^2} J( \gamma(t)) \right|_{t=0} = \left. d^2J(\gamma(t))[\gamma'(t),\gamma'(t)] \right|_{t=0} + \left. dJ(\gamma(t))[\gamma''(t)] \right|_{t=0} = d^2 J(\mf{u})[\mf{v},\mf{v}],
\]
which proves the claim \eqref{claim on second diff}. By direct computations, one can easily check that 
\begin{equation}\label{second diff}
d^2J(\mf{u})[\mf{v},\mf{v}] = \sum_{i=1}^d \|v_i\|_i^2 - \sum_{i,j=1}^d \int_{\Omega} \beta_{ij} u_i^2 v_j^2 -2 \sum_{i,j=1}^d \int_{\Omega} \beta_{ij} u_i u_j v_i v_j.
\end{equation}
We consider the variation $\mf{v}$ defined by 
\[
v_i:= \begin{cases} 0 & \text{if $i \neq l$} \\ u_p & \text{if $i = l$} \end{cases}.
\]
Since $\langle \nabla G_h(\mf{u}),\mf{v} \rangle = 0$ for every $h$, we have that $\mf{v} \in T_\mf{u}(\Ncal)$ (for the reader's convenience, we recall that $G_h$ has been defined in \eqref{def G_h}). Plugging this choice of $\mf{v}$ into \eqref{claim on second diff}, we infer that
\begin{equation}\label{eq48so}
\begin{split}
0 &\le \|u_p\|_{l}^2 - \sum_{i \in I_{\bar h}} \int_{\Omega} \beta_{il} u_i^2 u_p^2 - \sum_{i \not \in I_{\bar h}} \int_{\Omega} \beta_{il} u_i^2 u_p^2 \\
& \le \|u_p\|_{p}^2 - \sum_{i \in I_{\bar h} \setminus \{l\}} \int_{\Omega} \beta_{\bar h} u_i^2 u_p^2 - \sum_{i \not \in I_{\bar h}} \int_{\Omega} b u_i^2 u_p^2,
\end{split}
\end{equation}
where we used assumptions ($i$)-($iii$). On the other hand, testing the equation for $u_p$ against $u_p$ itself, and recalling that $u_l \equiv 0$, we deduce that
\begin{equation}\label{eq49so}
\begin{split}
\|u_p\|_{p}^2 &= \sum_{i \in I_{\bar h} \setminus \{l\}} \int_{\Omega} \beta_{ip} u_i^2 u_p^2 + \sum_{i \not \in I_{\bar h}} \int_{\Omega} \beta_{ip} u_i^2 u_j^2 \\
&=\int_{\Omega} \beta_{pp} u_p^4+  \sum_{i \in I_{\bar h} \setminus \{l,p\}} \int_{\Omega} \beta_{\bar h} u_i^2 u_p^2 + \sum_{i \not \in I_{\bar h}} \int_{\Omega} b u_i^2 u_p^2.
\end{split}
\end{equation}
Therefore, coming back to \eqref{eq48so}, we obtain
\[
0  \le \int_{\Omega} \beta_{pp} u_p^4 +  \sum_{i \in I_{\bar h} \setminus \{l,p\}} \int_{\Omega} \beta_{\bar h} u_i^2 u_p^2 - \sum_{i \in I_{\bar h} \setminus \{l\}} \int_{\Omega} \beta_{\bar h} u_i^2 u_p^2 = \int_{\Omega} \left(\beta_{pp}-\beta_{\bar h} \right) u_p^4 <0
\]
whenever $\beta_{\bar h} > \beta_{pp}$, which is guaranteed by our assumptions. $\qedhere$

\subsection{Proof of Theorem \ref{thm: weak coop and strong coop 2}}\label{subsec:second_new_results_positive}
As in the previous subsection, let $l\in I_{\bar h}$ with $u_l=0$. From Lemma \ref{lemma:lower_bounds} and since $\beta_{\bar h}>\max_{i
\in I_{\bar h}}\{\beta_{ii}\}$, there exists $p\in I_{\bar h}$ such that
\[
\beta_{\bar h} |u_p|_{L^4}^2 \geq \frac{\delta}{|I_{\bar h}|}\geq \frac{\delta}{d}.
\]
We have
\[
d^2 J(\mf{u})[\mf{v},\mf{v}] \ge 0 \quad \text{for every $\mf{v} \in T_\mf{u}(\Ncal)$},
\]
and take the admissible variation $\mf{v}$ defined  by 
\[
v_i:= \begin{cases} 0 & \text{if $i \neq l$} \\ u_p & \text{if $i = l$} \end{cases}
\]
Thus, as in \eqref{eq48so} and \eqref{eq49so}, we find
\[
0 \le \|u_p\|_{p}^2 - \sum_{i \in I_{\bar h} \setminus \{l\}} \int_{\Omega} \beta_{\bar h} u_i^2 u_p^2 - \sum_{i \not \in I_{\bar h}} \int_{\Omega} \beta_{il} u_i^2 u_p^2,
\]
and also
\[
\|u_p\|_{p}^2 =  \int_{\Omega} \beta_{pp} u_p^4+  \sum_{i \in I_{\bar h} \setminus \{l,p\}} \int_{\Omega} \beta_{\bar h} u_i^2 u_p^2 + \sum_{i \not \in I_{\bar h}} \int_{\Omega} \beta_{ip} u_i^2 u_p^2.
\]
Therefore, and recalling the explicit shapes $K=S^2/16\bar C$ and $\delta=S/2d$ from Lemma \ref{lemma:lower_bounds} and Remark \ref{rem: constant K}, 
\begin{align*}
0 & \le \int_{\Omega} \beta_{pp} u_p^4 +  \sum_{i \in I_{\bar h} \setminus \{l,p\}} \int_{\Omega} \beta_{\bar h} u_i^2 u_p^2 - \sum_{i \in I_{\bar h} \setminus \{l\}} \int_{\Omega} \beta_{\bar h} u_i^2 u_p^2 + \sum_{i \not \in I_{\bar h}} \int_{\Omega} ( \beta_{ip}-\beta_{il} ) u_i^2 u_p^2 \\
& = \int_{\Omega} \left(\beta_{pp}-\beta_{\bar h} \right) u_p^4 +  \sum_{i \not \in I_{\bar h}} \int_{\Omega} ( \beta_{ip}-\beta_{il} ) u_i^2 u_p^2 \\
& \le \underbrace{\left(\beta_{pp}-\beta_{\bar h} \right)}_{<0} |u_p|_{L^4}^4 + \frac{2K}{\alpha d^2 S}  |u_p|_{L^4}^2 \sum_{i \not \in I_{\bar h}} \|u_i\|_i^2 \\
&\le \frac{\beta_{pp}-\beta_{\bar h}}{\beta_{\bar h}} \frac{\delta}{d} |u_p|_{L^4}^2 + \frac{8K\bar C}{\alpha d^2 S}  |u_p|_{L^4}^2\\
& = \frac{S}{2d^2} \left(\frac{\beta_{pp}}{\beta_{\bar h}}-1+\frac{1}{\alpha}\right) |u_p|_{L^4}^2 = \frac{S}{2d^2} \left(\frac{\beta_{pp}}{\beta_{\bar h}}- \frac{\alpha-1}{\alpha}\right) |u_p|_{L^4}^2 <0,
\end{align*}
where we used the assumption on the coupling parameters, and the estimates of Lemmas \ref{eq:c_B_universalbound} and \ref{lemma:lower_bounds}.

\section{Non existence results in $H^1(\R^N)$}\label{sec: les space}

In this section we prove Theorem \ref{thm: les in space comp}, which illustrates that when working in $\Omega=\R^N$ in presence of simultaneous cooperation and competition, in order to find some kind of least energy solution it is often necessary to work in $H^1_{\rm rad}(\R^N)$ instead that in $H^1(\R^N)$. We choose a $m$-decomposition $\mf{a}$ of $d$, and we assume that the basic assumption \eqref{basic assumption in space} holds true. Throughout this section we assume that $\beta_{ij}>0$ for every $(i,j)\in I_h^2$ (recall the definition \eqref{def Ih} of $I_h$) and, for every $h=1,\dots,m$, we consider the sub-system
\begin{equation}\label{sub-sys}
\begin{cases}
-\Delta v_i + \lambda_i v_i = \sum_{j \in I_h} \beta_{ij} v_i v_j^2 & \text{in $\R^N$} \\
v_i \in H^1(\R^N),
\end{cases} \qquad \forall \ i \in I_h.
\end{equation}
We introduce the functional
\[
E_h(\mf{v}):= \int_{\R^N} \sum_{i \in I_h} \frac{1}{2} \left(|\nabla v_i|^2 + v_i^2\right) - \frac{1}{4}\sum_{(i,j) \in I_h^2} \beta_{ij} v_i^2 v_j^2,
\]
and the Nehari manifold for the system \eqref{sub-sys}, defined by
\[
\mathcal{M}_h := \left\{ \mf{v} \in (H^1(\R^N))^{a_h-a_{h-1}}: \text{$\mf{v} \neq \mf{0}$ and $\langle \nabla E_h(\mf{v}),\mf{v} \rangle = 0$} \right\}.
\]
We set 
\[
l_h:= \inf_{\mathcal{M}_h} E_h.
\]

The strategy consists in showing that $l$, defined in \eqref{def di l}, coincides with the sum of the least energy levels $l_h$ of the uncoupled sub-systems \eqref{sub-sys}. This is inspired by Theorem 1 in \cite{LinWei}, which is a particular case of our Theorem \ref{thm: les in space comp} (for $m=d$, $\Kcal_1=\emptyset$). We point out that our proof present substantial differences with respect to the one in \cite{LinWei}, referring to the forthcoming Remark \ref{rem: on differences in proofs} for more details. 

 Before proceeding, we need the following preliminary result. Although it is essentially known by the community, we present a short proof of it here, as we were not able to find any reference.

\begin{lemma}\label{lemma:exponentialdecay} Let $(u_i)_{i\in I_h}$ be a nonnegative solution of \eqref{sub-sys} with $u_i \in H^1(\R^N)$. Then for each $0<\beta<\min_i\{\lambda_i\}$ there exists $\alpha>0$ such that
\[
|u_i(x)|\leq \alpha e^{-\sqrt{1+\beta |x|^2}},\qquad \forall x\in \R^N, \ i\in I_h.
\]
\end{lemma}
\begin{proof} By a Brezis-Kato type argument, one has that $u_i \in L^\infty(\R^N)$. Thus, by standard gradient estimates for Poisson's equation (see \cite[eq. (3.15)]{Gilbargtrudinger}), we deduce also that $\nabla u_i \in L^\infty(\R^N)$. Since we also assume that $u_i\in L^2(\R^N)$, this clearly implies that $u_i\to 0$ as $|x|\to \infty$.

Defining $z(x)=\alpha \exp(-\sqrt{1+\beta |x|^2})$, a straightforward computation gives
\[
-\Delta z+\beta z\geq \frac{\alpha \beta}{1+\beta|x|^2} e^{-\sqrt{1+\beta |x|^2}}.
\]
The difference $z-u_i$ satisfies:
\[
-\Delta (z-u_i) + \beta (z-u_i)\geq (\lambda_i-\beta) u_i -\sum_{j=1}^k \beta_{ij} u_i u_j^2 + \frac{\alpha \beta}{1+\beta|x|^2} e^{-\sqrt{1+\beta |x|^2}},
\]
where $k=|I_h|$. Fix any $0<\beta<\lambda_i$. Since $\sum_{j=1}^k \beta_{ij} u_i u_j^2\to 0$ as $|x|\to \infty$, we can take $R>0$  such that
\[
\sum_{j=1}^k \beta_{ij} u_i u_j^2\leq (\lambda_i-\beta) u_i, \quad \text{ for } |x|\geq R,
\]
which implies that $-\Delta (z-u_i)+\beta (z-u_i)\geq 0$ for $|x|\geq R$. On the other hand, there exists $C>0$ and a sufficiently large $\alpha>0$, such that
\[
\sum_{j=1}^k \beta_{ij} u_i u_j^2 \leq C \leq \frac{\alpha \beta}{1+\beta R^2} e^{-\sqrt{1+\beta R^2}}\leq \frac{\alpha \beta}{1+\beta|x|^2} e^{-\sqrt{1+\beta |x|^2}} \quad \text{ for } |x| \le R.
\]
To sum up, we show that it is possible to choose $\alpha>0$ in such a way that
\[
-\Delta (z-u_i) + \beta (z-u_i)\geq 0 \qquad \text{ in } \R^N,
\]
and testing the inequality with $(z-u_i)^-$, we deduce that $u_i \leq z$.
\end{proof}

For $h=1,\dots,m$, we introduce 
\begin{align*}
\widetilde E_h(\mf{v}) &:= \frac{1}{4} \sum_{i \in I_h} \int_{\Omega} (|\nabla v_i|^2 + \lambda_i v_i^2) = \frac{1}{4} \|\mf{v}\|_h^2 \\
\widetilde{\mathcal{M}}_h& := \left\{ \mf{v}: \text{$\mf{v} \neq \mf{0}$ and } \|\mf{v}\|_h^2 \le \sum_{(i,j) \in I_h^2} \int_{\R^N} \beta_{ij} v_i^2 v_j^2 \right\} \\
\tilde l_h & := \inf_{\widetilde{\mathcal{M}}_h} \widetilde E_h.
\end{align*}

\begin{lemma}\label{lem: all achieved}
Both $l_h$ and $\tilde l_h$ are achieved, $l_h= \tilde l_h$, and any minimizer for $\tilde l_h$ is a minimizer for $l_h$. 
\end{lemma}
\begin{proof}
Since $\mathcal{M}_h \subset \widetilde{\mathcal{M}}_h$ and 
\[
E_h(\mf{v}) = \frac{1}{4} \|\mf{v}\|_h^2 = \widetilde E_h(\mf{v}) \qquad \forall \mf{v} \in \mathcal{M}_h,
\]
we have $\tilde l_h \le l_h$. As far as $\tilde l_h$ is concerned, we start by observing that, if $\mf{v} \in \widetilde{\mathcal{M}}_h$, then also its Schwarz symmetrization $\mf{v}^* \in \widetilde{\mathcal{M}}_h$, and by the Polya-Szego inequality $\|\mf{v}^*\|_h^2 \le \|\mf{v}\|_h^2$. Therefore, 
\[
\tilde l_h = \inf\left\{ \widetilde E_h(\mf{v}): \mf{v} \in \widetilde{\mathcal{M}}_h, \text{ $\mf{v}$ is radial}\right\}.
\] 
The functional $\widetilde E_h$ is coercive in $H^1_{\rad}(\R^N)$, so that any minimizing sequence is bounded from above. Reasoning exactly as in Lemma \ref{lemma:lower_bounds}, any such sequence is also bounded from below. This permits immediately to obtain the existence of a minimizer for $\tilde l_h$ (in this step it is used the fact that $H^1_{\rad}(\R^N)$ compactly embeds into $L^4(\R^N)$ for $N=2,3$). To complete the proof, we show that if $\tilde{\mf{v}}$ is a minimizer for $\tilde l_h$, then $\tilde{\mf{v}} \in \mathcal{M}_h$ and $\tilde E_h(\tilde{\mf{v}})=l_h$. Let $\Psi(t):= E_h(\sqrt{t} \tilde{\mf{v}})$. By definition we have that 
\[
t>0 \quad \text{and} \quad \Psi'(t) = 0 \quad \Longleftrightarrow \quad \sqrt{t} \tilde{\mf{v}} \in \tilde{\mathcal{M}}_h.
\]
By direct computations, it is easy to check that the unique positive critical point of $\Psi$ is given by
\[
\tilde t= \frac{\|\tilde{\mf{v}}\|_h^2}{ \sum_{(i,j) \in I_h^2}\int_{\R^N} \tilde v_i^2 \tilde v_j^2} \le 1,
\]
where the last estimate follows by the fact that $\tilde{\mf{v}} \in \tilde{\mathcal{M}}_h$. Thus, we have
\[
l_h \le E_h\left(\sqrt{\tilde t} \tilde{\mf{v}}\right) = \frac{1}{4} \|\tilde{\mf{v}}\|_h^2 \tilde t \le \frac{1}{4} \|\tilde{\mf{v}}\|_h^2 = \tilde l_h,
\]
which implies $l_h=\tilde l_h$, and in turn forces $\tilde t=1$, that is $\tilde{\mf{v}} \in \mathcal{M}_h$. 
\end{proof}

From now on, for each $h=1,\ldots, m$, we fix a minimizer $\mf{v}^h$ for $l_h$, hence a nontrivial solution of \eqref{sub-sys}. We have the following decay estimate.

\begin{lemma}\label{lem: decay}
Let $e_1 \neq e_2 \in \mathbb{S}^{N-1}$, and $\sigma_1,\sigma_2>0$. Then, whenever $h_1\neq h_2$,
\[
\lim_{R \to +\infty} \int_{\R^N}   \sum_{(i,j)\in I_{h_1} \times I_{h_2}} \left(v^{h_1}_i(x- R e_1) v^{h_2}_j(x-R e_2)\right)^2 \, dx = 0.
\]
\end{lemma}
\begin{proof}
Since $e_1 \neq e_2$, $R|\sigma_1 e_1- \sigma_2e_2| \to +\infty$ as $R \to +\infty$. Recalling from Lemma \ref{lemma:exponentialdecay} that each $v^h_i$ is exponentially decaying as $|x| \to +\infty$, the thesis follows easily.
\end{proof}

In the next lemma we show that the least energy level of the complete $d$ system \eqref{system}, which is denoted by $l$, can be controlled by the sum of the least energy levels $l_h$ of the sub-systems \eqref{sub-sys}.

\begin{lemma}
In the previous notation, we have $l \le \sum_{h=1}^m l_h$.
\end{lemma}
\begin{proof}
First of all, we observe that 
\[
J(\mf{u}) = \sum_{h=1}^m E_h(\mf{u}_h) + \sum_{(i,j) \in \mathcal{K}_2} \int_{\R^N} \beta_{ij} u_i^2 u_j^2.
\]
Let $e_1 \neq e_2 \neq \dots \neq e_m \in \mathbb{S}^{N-1}$ be $m$ different directions in $\R^N$. For $R>0$, we define $\mf{u}^R$ by means of
\[
u_i^R(x) := v_i^h(x-R e_h) \qquad \text{for $i \in I_h$, for $h=1,\dots,m$}.
\] 
Clearly, by a change of variables we have that 
\[
E_h(\mf{u}_h^R)= l_h \qquad \forall R>0.
\]
We aim at solving the linear system in $t_1,\dots,t_m$
\[
\frac{\partial}{\pa t_h} J\left(\sqrt{t_1} \mf{u}_1^R,\dots, \sqrt{t_m} \mf{u}_m^R\right) = 0 \quad \Longleftrightarrow \quad \sum_{k=1}^m M_B(\mf{u}^R)_{hk} t_k = \|\mf{u}^R_h\|_h^2.
\]
We claim that, for every $R \gg 1$ sufficiently large, this system has a solution $(t_1^R,\dots,t_m^R)$ with $0<t_h^R \to 1$ as $R \to +\infty$. Once that this is proved, we deduce that 
\begin{equation}\label{test on nehari}
\left(\sqrt{t_1^R} \mf{u}_1^R,\dots,\sqrt{t_m^R}\mf{u}_m^R\right) \in \mathcal{M}.
\end{equation}
To prove the claim, we observe that by Lemma \ref{lem: decay} $M_B(\mf{u}^R)_{hk} \to 0$ as $R \to +\infty$ for every $h \neq k$. Since on the contrary $M_{B}(\mf{u}^R)_{hh}$ is positive and constant in $R$, we deduce that $M_B(\mf{u}^R)$ is strictly diagonally dominant, and hence invertible, for every $R$ sufficiently large. As a consequence, for any such $R$ we can compute
\[
\left( \begin{array}{c} t_1^R \\ \vdots \\ t_m^R \end{array} \right) = M_B(\mf{u}^R)^{-1} \left( \begin{array}{c} \|\mf{u}_1^R\|_1^2 \\ \vdots \\ \|\mf{u}_m^R\|_m^2  \end{array} \right) = M_B(\mf{u}^R)^{-1} \left( \begin{array}{c} \|\mf{v}^1\|_1^2 \\ \vdots \\ \|\mf{v}^m\|_m^2  \end{array} \right),  
\]
But, as already observed, $M_B(\mf{u}^R)$ is converging to a diagonal matrix, whose diagonal entry in the $h$-th row is equal to 
\[
\sum_{(i,j) \in I_h^2} \int_{\R^N}  \left(v^h_i(x-R e_h) v^h_j(x-R e_h)\right)^2 \, dx = \sum_{(i,j) \in I_h^2} \int_{\R^N} (v_{i}^h v_{j}^h)^2,
\]
so that 
\[
\lim_{R \to +\infty} t_h^R = \frac{\|\mf{v}_h\|_h^2}{\sum_{(i,j) \in I_h^2} \int_{\R^N} (v_{i}^h v_{j}^h)^2} = 1,
\]
where we used the fact that by assumption $\mf{v}^h \in \mathcal{M}_h$. 

To sum up, we have just showed that for any large $R$ there exists $\mf{t}^R \simeq \mf{1}$ such that \eqref{test on nehari} holds. Therefore
\begin{align*}
l &\le \lim_{R \to +\infty} J \left(\sqrt{t_1^R} \mf{u}_1^R,\dots,\sqrt{t_m^R}\mf{u}_m^R\right) \\
&= \lim_{R \to +\infty} \sum_{h=1}^m E_h\left( \sqrt{t_h^R}\mf{u}_h^R\right) + \lim_{R \to +\infty} \sum_{h \neq k} M_B \left(\sqrt{t_1^R} \mf{u}_1^R,\dots,\sqrt{t_m^R}\mf{u}_m^R\right)  \\
&= \sum_{h=1}^m E_h\left( \mf{v}^h\right) = \sum_{h=1}^m l_h. \qedhere
\end{align*}
\end{proof}

Now we have to show that the opposite inequality holds.

\begin{lemma}
There holds $l \ge \sum_{h=1}^m l_h$.
\end{lemma}
\begin{proof}
Let $\mf{u} \in \mathcal{M}$. Thanks to the assumption $\beta_{ij} \le 0$ for every $(i,j) \in \mathcal{K}_2$, we have
\[
0<\|\mf{u}_h\|_h^2 = \sum_{k=1}^m M_B(\mf{u})_{hk} \le M_B(\mf{u})_{hh},
\]
so that $\mf{u}_h \in \widetilde{\mathcal{M}}_h$ for every $h=1,\dots,m$. As a consequence
\begin{equation}\label{eq1}
\frac{1}{4}\|\mf{u}_h\|_h^2 \ge \inf_{\mf{v} \in \widetilde{\mathcal{M}}_h} \frac{1}{4} \|\mf{v}\|_h^2=\tilde l_h=l_h, \qquad  \text{ and }\qquad  J(\mf{u}) = \sum_{h=1}^m \frac{1}{4}\|\mf{u}_h\|_h^2 \ge \sum_{h=1}^m l_h. \qedhere
\end{equation}
\end{proof}

\begin{proof}[Conclusion of the proof of Theorem \ref{thm: les in space comp}]
By contradiction, we suppose that there exists $\mf{u} \in \mathcal{M}$ such that $J(\mf{u}) =l$. Thanks to the fact that $\mathcal{M}$ is a natural constraint (cf. Lemma \ref{lemma:N_is_manifold}, which holds also in $H^1(\R^N)$), $\mf{u}$ is a solution of \eqref{system}, which we can assume to be nonnegative. By the strong maximum principle and the definition of $\mathcal{M}$, we deduce that for every $h$ there exists $i_h$ such that $u_{i_h}>0$ in $\R^N$. To reach a contradiction, we observe that the first equation in \eqref{eq1}, together with the fact that $\sum_h l_h=l$, imply that necessarily
\begin{equation}\label{eq2}
\frac{1}{4}\|\mf{u}_h\|_h^2 = l_h = \tilde l_h.
\end{equation}
Since $\beta_{ij} \le 0$ for every $(i,j) \in \mathcal{K}_2$, we have $\mf{u}_h \in \widetilde{\mathcal{M}}_h$, so that $\mf{u}_h$ minimizes $\tilde l_h=l_h$. Thus, by the last statement of Lemma \ref{lem: all achieved}, also $\mf{u}_h\in \Mcal_h$, and in particular
\[
\|\mf{u}_h\|_h^2 = M_B(\mf{u})_{hh}\qquad \forall h=1,\ldots, m.
\]
This is in contradiction with the fact that, since $\mf{u}$ was supposed to be in $\mathcal{M}$, we have
\[
\|\mf{u}_{h_1}\|_{h_1}^2 = \sum_{k=1}^m M_B(\mf{u})_{h_1 k} \le M_B(\mf{u})_{h_1 h_1} + \int_{\R^N} \beta_{i_{h_1} i_{h_2}} u_{i_{h_1}}^2 u_{i_{h_2}}^2 < M_B(\mf{u})_{h_1 h_1},
\]
(for the reader's convenience we recall that $h_1$ and $h_2$ have been introduce in the assumptions of Theorem \ref{thm: les in space comp}).
\end{proof}

\begin{remark}\label{rem: on differences in proofs} 
A remarkable fact, which marks a significant difference in our proof with respect to that of Theorem 1 in \cite{LinWei}, is that we do not assume that each $l_h$ admits a unique minimizer, nor that it is achieved by a unique positive solution, of the form $u_i(x)=\alpha_i w(\sigma_i x)$ (with $w$ the unique positive radially decreasing solution of the related single equation problem). This was used in \cite{LinWei}. Instead, our argument is only based upon the decay estimate provided by Lemma \ref{lemma:exponentialdecay}. 

Concerning the uniqueness of ground states, although sufficient conditions that imply uniqueness are already known in the literature, it is still an open problem to completely determine the range of parameters for which completely cooperative systems as \eqref{sub-sys} have a unique solution, corresponding to the least energy positive level. This is known in the $2$ component case with $\lambda_1=\lambda_2$ and $\beta_{12}>\max\{\beta_{11},\beta_{22}\}$, see \cite{WeiYao}. 
For systems of more than $2$ components, we refer to Theorem 1.1 in \cite{ChangLiuLiu}, where it is shown in particular that if $\lambda_i\equiv \lambda$ and $\beta_{ij}$, $i\neq j$ are large and satisfy additional technical assumptions, then any sub-system \eqref{sub-sys} has a unique least energy positive solution. The results in \cite[Section 2]{LiuWang}, in \cite[Section 4]{Sirakov}, and \cite[Proposition 2.1]{Bartsch} suggest that uniqueness should hold also in more general situations.
\end{remark}


\end{document}